\documentclass[reqno]{amsart}
\usepackage{amssymb,mathtools,lineno}
\allowdisplaybreaks[4]

\usepackage{ifpdf}
\ifpdf
 \usepackage[hyperindex]{hyperref}%,pagebackref
\else
 \expandafter\ifx\csname dvipdfm\endcsname\relax
 \usepackage[hypertex,hyperindex]{hyperref}
 \else
 \usepackage[dvipdfm,hyperindex]{hyperref}
 \fi
\fi

\theoremstyle{plain}
\newtheorem{thm}{Theorem}
\newtheorem{cor}{Corollary}

\theoremstyle{remark}
\newtheorem{rem}{Remark}

\DeclareMathOperator{\td}{d\!}
\DeclareMathOperator{\te}{e}
\DeclareMathOperator{\ti}{i}
\DeclareMathOperator{\arcsinh}{arcsinh}
\DeclareMathOperator{\arctanh}{arctanh}

\numberwithin{equation}{section}

\begin{document}

\title[Formulas of series and closed forms of hypergeometric functions]
{Infinite and finite series involving central binomial coefficients and closed forms of generalized hypergeometric functions}

\author[G. B. Basnet]{Ganesh Bahadur Basnet}
\address{Department of Mathematics, Tri-Chandra Multiple Campus, Tribhuvan University, Kathmandu, Nepal}
\email{gbbmath@gmail.com}
\urladdr{\url{https://orcid.org/0009-0002-7721-9599}}

\author[N. P. Pahari]{Narayan Prasad Pahari}
\address{Central Department of Mathematics, Tribhuvan University, Kathmandu, Nepal}
\email{nppahari@gmail.com}
\urladdr{\url{https://orcid.org/0009-0006-8814-2151}}

\author[A. K. Rathie]{Arjun Kumar Rathie}
\address{Department of Mathematics, Vedant College of Engineering and Technology, Rajasthan Technical University, Tulsi, Jakhamund, Bundi, Rajasthan State, India}
\email{arjunkumarrathie@gmail.com}
\urladdr{\url{https://orcid.org/0000-0003-3902-3050}}

\author[F. Qi]{Feng Qi*}
\address{School of Mathematics and Physics, Hulunbuir University, Hulunbuir 021008, Inner Mongolia, China;
17709 Sabal Court, University Village, Dallas, TX 75252-8024, USA}
\email{\href{mailto:qifeng618@gmail.com}{qifeng618@gmail.com}}
\urladdr{\url{https://orcid.org/0000-0001-6239-2968}}

\begin{abstract}
Let $\mathbb{Z}^-=\{-1,-2,\dotsc\}$ and $\mathbb{N}=\{1,2,\dotsc\}$. In 2023, Qi and Lim gave two claims for summing the infinite series $\sum_{k=1}^{\infty}\binom{2k}{k}\frac{1}{\alpha+k}\bigl(\frac{\pm1}{4}\bigr)^k$ for $\alpha\in\mathbb{C}\setminus\mathbb{Z}^-$. In this paper, we confirm these claims by expressing the infinite series $\sum_{k=1}^{\infty}\binom{2k}{k}\frac{1}{\alpha+k}\bigl(\frac{z}{4}\bigr)^k$ and the finite series $\sum_{k=1}^{n}\binom{2k}{k}\frac{1}{\alpha+k}\bigl(\frac{z}{4}\bigr)^k$ in terms of the Gauss hypergeometric function ${}_2F_1\bigl[\tfrac12,\alpha;1+\alpha;z\bigr]$ and a generalized hypergeometric function ${}_3F_2$, for $|z|\le1$ and $n\in\mathbb{N}$. Using the Euler integral representation, we derive closed forms, in terms of finite sums of sines and cosines of $\arcsin\sqrt{z}\,$, for ${}_2F_1\bigl[\tfrac12,\tfrac12+n;\tfrac32+n;z\bigr]$, ${}_2F_1\bigl[\tfrac12,1+n;2+n;z\bigr]$, two related ${}_3F_2$ functions, and the incomplete beta functions $B_z\bigl(\frac12+n,\frac12\bigr)$ and $B_z\bigl(1+n,\frac12\bigr)$. By the Pfaff transformation, we obtain closed forms of several further Gauss hypergeometric functions with parameters shifted by $n$. We also give a closed form of $\bigl[(1-z)\frac{\td}{\td z}(1-z)\bigr]^n\frac{\arcsin\sqrt{z}\,}{\sqrt{z(1-z)}\,}$ and a combinatorial identity.
\end{abstract}



\thanks{*Corresponding author: Feng Qi, qifeng618@gmail.com}

\thanks{This paper was typeset using \AmS-\LaTeX}

\maketitle

\section{Introduction and main results}
The generalized hypergeometric function ${\,}_pF_q$ is defined~\cite{a2000, b1935, r1960, s1966} by
\begin{equation}\label{pFq}
{\,}_pF_q
\begin{bmatrix}
\begin{gathered}
a_1, a_2, \dotsc, a_p\\
b_1, b_2, \dotsc, b_q
\end{gathered}; z
\end{bmatrix}
=\sum_{n=0}^{\infty} \frac{(a_1)_n(a_2)_n \dotsm (a_p)_n}{(b_1)_n (b_2)_n \dotsm (b_q)_n} \frac{z^n}{n!}
\end{equation}
for $p,q\in\mathbb{N}_0=\{0,1,2,\dotsc\}$ and $b_j\in\mathbb{C}\setminus\{0,-1,-2,\dotsc\}$, where the Pochhammer symbol $(z)_n$ is defined as
\begin{equation*}
(z)_{n}
= 
\begin{dcases}
z(z+1)\dotsm(z+n-1), & n \in \mathbb{N}=\{1,2,\dotsc\}\\
1, & n=0
\end{dcases}
\end{equation*}
for $z\in\mathbb{C}$, which equals $\Gamma(z+n)/\Gamma(z)$ for $z\in\mathbb{C}\setminus\{0,-1,-2,\dotsc\}$, and the classical Euler gamma function $\Gamma(z)$ can be defined~\cite[Chapter~3]{Temme-96-book} by
\begin{equation*}
\Gamma(z)=\lim_{n\to\infty}\frac{n!n^z}{\prod_{k=0}^n(z+k)}, \quad z\in\mathbb{C}\setminus\{0,-1,-2,\dotsc\}.
\end{equation*}
The generalized hypergeometric series~\eqref{pFq} is convergent for all $|z|<\infty$ if $p \le q$ and for $|z|<1$ if $p=q+1$, while it is divergent for all $z\ne0$ if $p>q+1$. When $|z|=1$ and $p=q+1$, the series~\eqref{pFq} converges absolutely if $\Re\bigl(\sum_{i=1}^{q} b_i-\sum_{i=1}^{p} a_i \bigr)>0$, converges conditionally if $-1<\Re\bigl(\sum_{i=1}^{q} b_i-\sum_{i=1}^{p} a_i \bigr) \le 0$ and $z\ne1$, and diverges if $\Re\bigl(\sum_{i=1}^{q} b_i-\sum_{i=1}^{p} a_i \bigr)\le-1$ (or if $z=1$ and $\Re\bigl(\sum_{i=1}^{q} b_i-\sum_{i=1}^{p} a_i \bigr)\le0$). In particular, for the case $p=2$ and $q=1$ in~\eqref{pFq}, the series ${\,}_2F_1$ is called the Gauss hypergeometric function in the literature. For more information on the Gauss hypergeometric function ${\,}_2F_1$, please refer to~\cite[Chapter~15]{NIST-HB-2010}, \cite[Chapter~5]{Temme-96-book}, and the recent papers~\cite{axioms-2962911.tex, Gauss-Milovanovic-Qi.tex, Gauss-hyperg-Int.tex}.
It is common knowledge that the generalized hypergeometric function ${\,}_pF_q$ occurs in numerous practical and theoretical applications such as mathematical physics, theoretical physics, engineering, statistics, and combinatorics.
\par
The classical Euler beta function $B(p,q)$ is defined~\cite[p.~258]{abram} by
\begin{equation*}
B(p,q)=\int_0^1t^{p-1}(1-t)^{q-1}\td t, \quad \Re(p),\Re(q)>0.
\end{equation*}
It can be expressed in terms of the gamma function $\Gamma(z)$ by
\begin{equation*}
B(p,q)=\frac{\Gamma(p)\Gamma(q)}{\Gamma(p+q)}, \quad \Re(p),\Re(q)>0.
\end{equation*}
The incomplete beta function is defined~\cite[p.~128]{Temme-96-book} by
\begin{equation*}
B_z(p,q)=\int_{0}^{z}t^{p-1}(1-t)^{q-1}\td t
\end{equation*}
for $0\le z\le1$ and $\Re(p),\Re(q)>0$, with $B_1(p,q)=B(p,q)$.
\par
In mathematics, a closed form is a mathematical expression that can be evaluated in a finite number of operations. It may contain constants, variables, four arithmetic operations, and elementary functions, but usually no limit. What are closed forms, and why do we care about them? Please refer to~\cite{closed-form-what-why-care} and references cited therein.
\par
In~\cite[Remarks~1 and~2]{q2022}, using the software \textsc{Wolfram Mathematica}, Qi and Lim calculated the concrete values for the infinite series
\begin{equation}\label{Qi-infi-ser}
\sum_{k=1}^{\infty} \binom{2k}{k}\frac{1}{m+k} \frac{(\pm1)^k}{4^k}, \quad 1\le m\le16.
\end{equation}
In~\cite[Remark~3]{q2022}, the sums
\begin{equation}\label{anonymous1}
\sum_{k=1}^{\infty}\binom{2k}{k}\frac{1}{\alpha+k}\frac{1}{4^k} =\frac{\sqrt{\pi}\,\Gamma(\alpha)}{\Gamma\bigl(\frac{1}{2}+\alpha\bigr)}-\frac{1}{\alpha}
\end{equation}
and
\begin{multline}\label{anonymous2}
\sum_{k=1}^{n}\binom{2k}{k}\frac{1}{\alpha+k}\frac{1}{4^k} =\frac{\sqrt{\pi}\,\Gamma(\alpha)}{\Gamma\bigl(\frac{1}{2}+\alpha\bigr)}-\frac{1}{\alpha}\\
-\frac{1}{1+\alpha+n}\frac{(1+2n)!!}{(2+2n)!!}\,{\,}_3F_2\begin{bmatrix}\begin{gathered}1,\tfrac{3}{2}+n,1+\alpha+n\\
2+n,2+\alpha+n
\end{gathered};1
\end{bmatrix}
\end{multline}
were claimed to be valid for $\alpha\in\mathbb{C}\setminus\mathbb{Z}^-$ and $n\in\mathbb{N}$, where $\mathbb{Z}^-=\{-1,-2,\dotsc\}$ and the right-hand sides at $\alpha=0$ are understood as limits as $\alpha\to0$.
\par
The aim of this paper is to positively confirm the validity of the sums~\eqref{anonymous1} and~\eqref{anonymous2}.
\par
Our main results in this paper are stated as follows.

\begin{thm}\label{Arjun-Gauss-thm}
For $\alpha\in\mathbb{C}\setminus\mathbb{Z}^-$ and $|z|\le1$, we have
\begin{align}\label{equ:ii}
\sum_{k=1}^{\infty}\binom{2k}{k}\frac{1}{\alpha+k}\biggl(\frac{z}{4}\biggr)^k
&=
\begin{dcases}\frac{1}{\alpha}\biggl({\,}_2F_1
\begin{bmatrix}
\begin{gathered}
\tfrac{1}{2}, \alpha\\
 1+\alpha
\end{gathered}
;z 
\end{bmatrix}
-1\biggr), & \alpha\ne0\\
2\ln\frac{2}{1+\sqrt{1-z}\,}, & \alpha=0
\end{dcases}
\intertext{and}\label{equ:i}
\sum_{k=1}^{\infty}\binom{2k}{k}\frac{1}{\alpha+k}\biggl(\frac{z}{4}\biggr)^k 
&=\frac{z}{2(1+\alpha)}{\,}_3F_2
\begin{bmatrix}
\begin{gathered}
 1,\tfrac{3}{2}, 1+\alpha\\
2, 2+\alpha
\end{gathered}
;z 
\end{bmatrix}.
\end{align}
\end{thm}

\begin{thm}\label{Finite-sum-thm}
For $\alpha\in\mathbb{C}\setminus\mathbb{Z}^-$, $|z|\le1$, and $n\in\mathbb{N}$, we have
\begin{equation}\label{equ:B}
\begin{aligned}
\sum_{k=1}^{n}\binom{2k}{k}\frac{1}{\alpha+k}\biggl(\frac{z}{4}\biggr)^k
&=-\frac{(1+2n)!!}{(2+2n)!!}\frac{z^{1+n}}{1+\alpha+n}
{\,}_3F_2
\begin{bmatrix}
\begin{gathered}
 1, \tfrac{3}{2}+n, 1+\alpha+n\\
2+n, 2+\alpha+n
\end{gathered}
;z
\end{bmatrix}\\
&\quad+\begin{dcases}\frac{1}{\alpha}\biggl({\,}_2F_1
\begin{bmatrix}
\begin{gathered}
\tfrac{1}{2}, \alpha\\
 1+\alpha
\end{gathered}
;z 
\end{bmatrix}
-1\biggr), & \alpha\ne0;\\
2\ln\frac{2}{1+\sqrt{1-z}\,}, & \alpha=0.
\end{dcases}
\end{aligned}
\end{equation}
\end{thm}

In what follows we state corollaries of Theorems~\ref{Arjun-Gauss-thm} and~\ref{Finite-sum-thm} and of the Euler integral representation. Throughout, $\sqrt{z}\,$ denotes the principal square root, $z^{n+1/2}:=(\sqrt{z}\,)^{2n+1}$, and $\arcsin$ is the principal branch. In particular, for $-1\le z<0$ we have $\sqrt{z}=\ti\sqrt{|z|}\,$, so $z^{n+1/2}=(-1)^n\ti\,|z|^{n+1/2}$ and $\arcsin\sqrt{z}=\ti\arcsinh\sqrt{|z|}\,$. Since each right-hand side below is even in $\sqrt{z}\,$, it is a single-valued analytic function of $z$ on $\mathbb{C}\setminus[1,\infty)$, so the branch of $\sqrt{z}\,$ is immaterial and the formulas in Corollaries~\ref{2F1-sum-ID}--\ref{3F2-pm1-cor} extend from $0<z\le1$ to $0<|z|\le1$ by analytic continuation.

\begin{cor}\label{2F1-sum-ID}
For $n\in\mathbb{N}_0$ and $0<|z|\le1$, we have two closed-form expressions
\begin{multline}\label{2F1n+half-2}
{\,}_2F_1
\begin{bmatrix}
\begin{gathered}
\tfrac{1}{2}, \tfrac{1}{2}+n\\
\tfrac{3}{2}+n
\end{gathered}
;z
\end{bmatrix}
=\frac{1+2n}{4^{n}z^{n+1/2}}\Biggl[\binom{2n}{n} \arcsin\sqrt{z}\,\\
+(-1)^n\Biggl(\sum_{k=0}^{n-1}
+\sum_{k=1+n}^{2n}\Biggr) (-1)^{k}\binom{2n}{k} \frac{\sin[2(n-k)\arcsin\sqrt{z}\,]}{2(n-k)}\Biggr],
\end{multline}
\begin{multline}\label{2F1=1+n}
{\,}_2F_1
\begin{bmatrix}
\begin{gathered}
\tfrac{1}{2}, 1+n\\
2+n
\end{gathered}
;z
\end{bmatrix}\\
=(-1)^{1+n}\frac{1+n}{4^{n}z^{1+n}}\sum_{k=0}^{1+2n}(-1)^{k} \binom{1+2n}{k} \frac{\cos[(2n-2k+1)\arcsin\sqrt{z}\,]-1}{2n-2k+1},
\end{multline}
and two identities
\begin{align}\label{ID=01}
\sum_{k=0}^{1+2n}(-1)^{k} \binom{1+2n}{k} \frac{\sin[(2n-2k+1)\arcsin\sqrt{z}\,]}{2n-2k+1}&=0,\\
\Biggl(\sum_{k=0}^{n-1}+\sum_{k=1+n}^{2n}\Biggr) (-1)^{k}\binom{2n}{k} \frac{1-\cos[2(n-k)\arcsin\sqrt{z}\,]}{n-k}&=0.
\label{ID=02}
\end{align}
\end{cor}

\begin{cor}\label{(-1)series-cor}
For $n\in\mathbb{N}_0$ and $0<|z|\le1$, we have
\begin{multline}\label{closed=int}
\sum_{k=1}^{\infty}\binom{2k}{k}\frac{1}{2k+2n+1}\biggl(\frac{z}{4}\biggr)^k
=\frac{1}{4^{n}z^{n+1/2}}\Biggl[\binom{2n}{n} \arcsin\sqrt{z}\,\\
+(-1)^n\Biggl(\sum_{k=0}^{n-1}
+\sum_{k=1+n}^{2n}\Biggr) (-1)^{k}\binom{2n}{k} \frac{\sin[2(n-k)\arcsin\sqrt{z}\,]}{2(n-k)}\Biggr]
-\frac{1}{1+2n}
\end{multline}
and
\begin{multline}\label{closed=integer}
\sum_{k=1}^{\infty}\binom{2k}{k}\frac{1}{1+k+n}\biggl(\frac{z}{4}\biggr)^k\\
=
\frac{(-1)^{1+n}}{4^{n}z^{1+n}} \sum_{k=0}^{1+2n}(-1)^{k} \binom{1+2n}{k} \frac{\cos[(2n-2k+1)\arcsin\sqrt{z}\,]-1}{2n-2k+1}-\frac{1}{1+n}.
\end{multline}
\end{cor}

\begin{cor}\label{3F2-pm1-cor}
For $n\in\mathbb{N}_0$ and $0<|z|\le1$, we have two closed-form expressions
\begin{multline}\label{3F2-closed1}
{\,}_3F_2
\begin{bmatrix}
\begin{gathered}
 1,\tfrac{3}{2}, \tfrac{3}{2}+n\\
2, \tfrac{5}{2}+n
\end{gathered}
;z 
\end{bmatrix}
=\frac{3+2n}{1+2n}\frac{2}{z}\Biggl(\frac{1+2n}{4^{n}} \frac{1}{z^{n+1/2}}\Biggl[\binom{2n}{n} \arcsin\sqrt{z}\,\\
+(-1)^n\Biggl(\sum_{k=0}^{n-1}
+\sum_{k=1+n}^{2n}\Biggr) (-1)^{k}\binom{2n}{k} \frac{\sin[2(n-k)\arcsin\sqrt{z}\,]}{2(n-k)}\Biggr]-1\Biggr)
\end{multline}
and
\begin{equation}
\begin{aligned}\label{3F2-closed2}
{\,}_3F_2
\begin{bmatrix}
\begin{gathered}
 1,\tfrac{3}{2}, 2+n\\
2, 3+n
\end{gathered}
;z 
\end{bmatrix}&=\frac{2(2+n)}{z}\Biggl[\frac{(-1)^{1+n}}{4^{n}}\frac{1}{z^{1+n}}\sum_{k=0}^{1+2n}(-1)^{k} \binom{1+2n}{k} \\
&\quad\times\frac{\cos[(2n-2k+1)\arcsin\sqrt{z}\,]-1}{2n-2k+1}-\frac{1}{1+n}\Biggr].
\end{aligned}
\end{equation}
\end{cor}

\begin{cor}\label{Beta-Closed-Cor}
For $n\in\mathbb{N}_0$ and $0\le z\le1$, we have the closed-form expressions
\begin{multline}\label{beta-closed1}
B_z\biggl(\frac{1}{2}+n,\frac12\biggr)
=\frac{1}{2^{2n-1}}\Biggl[\binom{2n}{n} \arcsin\sqrt{z}\,\\
+(-1)^{n}\Biggl(\sum_{k=0}^{n-1}+\sum_{k=1+n}^{2n}\Biggr) (-1)^{k}\binom{2n}{k} \frac{\sin[2(n-k)\arcsin\sqrt{z}\,]}{2(n-k)}\Biggr]
\end{multline}
and
\begin{multline}\label{beta-closed2}
B_z\biggl(1+n,\frac12\biggr)\\*
=\frac{(-1)^{1+n}}{4^{n}}\sum_{k=0}^{1+2n}(-1)^{k} \binom{1+2n}{k} \frac{\cos[(2n-2k+1)\arcsin\sqrt{z}\,]-1}{2n-2k+1}.
\end{multline}
\end{cor}

\begin{cor}\label{derv-form-cor}
For $n\in\mathbb{N}_0$ and $0<z<1$, we have the derivative formula
\begin{multline}\label{deriv-form}
\biggl[(1-z)\frac{\td}{\td z}(1-z)\biggr]^n \frac{\arcsin\sqrt{z}\,}{\sqrt{z(1-z)}\,}
=\frac{n!}{4^{n}} \frac{(1-z)^{n-1/2}}{z^{n+1/2}}\Biggl[(-1)^n\binom{2n}{n} \arcsin\sqrt{z}\,\\
+\Biggl(\sum_{k=0}^{n-1} +\sum_{k=1+n}^{2n}\Biggr) (-1)^{k}\binom{2n}{k} \frac{\sin[2(n-k)\arcsin\sqrt{z}\,]}{2(n-k)}\Biggr].
\end{multline}
\end{cor}

In the next section, we prove these theorems and corollaries in detail.

\section{Proofs of two theorems and five corollaries}
We now provide detailed proofs of two theorems and five corollaries.

\begin{proof}[Proof of Theorem~\ref{Arjun-Gauss-thm}]
Let
\begin{equation}\label{S(alpha)}
S(\alpha)=\sum_{k=1}^{\infty}\binom{2k}{k}\frac{1}{\alpha+k}\biggl(\frac{z}{4}\biggr)^k, \quad \alpha\in\mathbb{C}\setminus\mathbb{Z}^-.
\end{equation}
Since
$$
\binom{2k}{k}=4^k\frac{\bigl(\frac {1}{2}\bigr)_k}{(1)_k}\quad\text{and}\quad \alpha+k =\frac{\alpha(1+\alpha)_k}{(\alpha)_k}
$$
for $k\in\mathbb{N}$ and $\alpha\in\mathbb{C}\setminus\{0,-1,-2,\dotsc\}$, we obtain
\begin{equation*}
S(\alpha)=\frac{1}{\alpha}\sum_{k=1}^{\infty}\frac{\bigl(\frac{1}{2}\bigr)_k (\alpha)_k}{(1+\alpha)_k} \frac{z^k}{k!}
=\frac{1}{\alpha} 
\biggl({\,}_2F_1
\begin{bmatrix}
\begin{gathered}
\tfrac{1}{2}, \alpha\\
 1+\alpha
\end{gathered}
;z 
\end{bmatrix}-1\biggr)
\end{equation*}
for $\alpha\in\mathbb{C}\setminus\{0,-1,-2,\dotsc\}$. The series ${\,}_2F_1
\begin{bmatrix}
\begin{gathered}
\tfrac{1}{2}, \alpha\\
 1+\alpha
\end{gathered}
;z 
\end{bmatrix}$ converges absolutely for $|z|\le1$,
since $|z|<1$ is covered by the radius of convergence and, for $|z|=1$,
$\Re\bigl((1+\alpha)-\tfrac12-\alpha\bigr)=\tfrac12>0$; hence the identity holds on the
closed unit disk. (This is the criterion for $|z|=1$ recalled after~\eqref{pFq}.) The formula~\eqref{equ:ii} for the case $\alpha\ne0$ thus follows.
\par
For $\alpha=0$, the formula~\eqref{equ:ii} follows from replacing $x^2$ by $-z$ in the series expansion
\begin{equation}\label{Grad-p.1.515}
\ln\bigl(1+\sqrt{1+x^2}\,\bigr) =\ln2-\sum_{k=1}^{\infty}(-1)^k\frac{(2k-1)!}{(k!)^2}\biggl(\frac{x^2}{4}\biggr)^{k}, \quad |x|\le1,
\end{equation}
see~\cite[p.~54, Entry~1.515]{Gradshteyn-Ryzhik-Table-8th}, which is valid for real $|x|\le1$, that is, for $-1\le z\le0$. Both sides of the resulting identity are analytic in $|z|<1$ (the left side because $\sqrt{1-z}\ne-1$ there, so the logarithm has no branch point in the disk), hence the identity extends to $|z|<1$ by analytic continuation, and to $|z|=1$ by Abel's theorem and continuity of the left-hand side.
\par
Using the identity $(a)_{1+k}=a(1+a)_k$, we derive
\begin{align*}
S(\alpha)&=\frac{1}{\alpha}\sum_{k=0}^{\infty}\frac{\bigl(\frac{1}{2}\bigr)_{1+k} (\alpha)_{1+k}}{(1)_{1+k} (1+\alpha)_{1+k}} z^{1+k}\\
&=\frac{z}{2(1+\alpha)}\sum_{k=0}^{\infty}\frac{\bigl(\frac{3}{2}\bigr)_k (1+\alpha)_k}{(2)_k (2+\alpha)_k} z^k\\
&=\frac{z}{2(1+\alpha)}\sum_{k=0}^{\infty}\frac{(1)_k\bigl(\frac{3}{2}\bigr)_k (1+\alpha)_k}{(2)_k (2+\alpha)_k} \frac{z^k}{k!}\\
&=\frac{z}{2(1+\alpha)} {\,}_3F_2
\begin{bmatrix}
\begin{gathered}
 1,\tfrac{3}{2}, 1+\alpha\\
2, 2+\alpha
\end{gathered}
;z 
\end{bmatrix}
\end{align*}
for $\alpha\ne0$. On the other hand, for $\alpha=0$, the formula~\eqref{equ:i} becomes
\begin{equation}\label{alpha=0}
\sum_{k=1}^{\infty}\binom{2k}{k}\frac{1}{k}\biggl(\frac{z}{4}\biggr)^k 
=\frac{z}{2}{\,}_3F_2
\begin{bmatrix}
\begin{gathered}
 1,\tfrac{3}{2}, 1\\
2, 2
\end{gathered}
;z 
\end{bmatrix}.
\end{equation}
Since
\begin{equation*}
{\,}_3F_2
\begin{bmatrix}
\begin{gathered}
 1,\tfrac{3}{2}, 1\\
2, 2
\end{gathered}
;z 
\end{bmatrix}
=-\frac{4}{z}\ln\frac{1+\sqrt{1-z}\,}{2},
\end{equation*}
see~\cite[p.~519, Entry~365]{p1990}, replacing $x^2$ by $-z$ in the series expansion~\eqref{Grad-p.1.515}, we acquire~\eqref{alpha=0}.
The formula~\eqref{equ:i} thus follows. The proof of Theorem~\ref{Arjun-Gauss-thm} is complete.
\end{proof}

\begin{proof}[Proof of Theorem~\ref{Finite-sum-thm}]
It is obvious that the quantity $S(\alpha)$ defined by~\eqref{S(alpha)} can be written as
\begin{align*}
S(\alpha)&=\sum_{k=1}^{n}\binom{2k}{k}\frac{1}{\alpha+k}\biggl(\frac{z}{4}\biggr)^k
+\sum_{k=1+n}^{\infty}\binom{2k}{k}\frac{1}{\alpha+k}\biggl(\frac{z}{4}\biggr)^k.
\end{align*} 
A straightforward computation gives
\begin{gather*}
\sum_{k=1+n}^{\infty}\binom{2k}{k}\frac{1}{\alpha+k}\biggl(\frac{z}{4}\biggr)^k
=\sum_{k=1+n}^{\infty}\frac{\bigl(\frac{1}{2}\bigr)_k}{\alpha+k}\frac{z^k}{k!}
=\sum_{k=0}^{\infty}\frac{\bigl(\frac{1}{2}\bigr)_{1+k+n}}{k+\alpha+n+1}\frac{z^{1+k+n}}{(1+k+n)!}\\
=\frac{\Gamma\bigl(\frac{3}{2}+n\bigr)}{\Gamma\bigl(\frac{1}{2}\bigr)} \frac{z^{1+n}}{(1+\alpha+n)\Gamma(2+n)}
{\,}_3F_2
\begin{bmatrix}
\begin{gathered}
 1,\tfrac{3}{2}+n, 1+\alpha+n\\
 2+n, 2+\alpha+n
\end{gathered}
;z 
\end{bmatrix},
\end{gather*}
where we used
\begin{gather*}
k+\alpha+n+1=(1+\alpha+n)\frac{(2+\alpha+n)_k}{(1+\alpha+n)_k},\quad
\biggl(\frac{1}{2}\biggr)_{1+k+n}=\frac{\Gamma\bigl(\frac{3}{2}+n\bigr)}{\Gamma\bigl(\frac{1}{2}\bigr)}\biggl(\frac{3}{2}+n\biggr)_k,\\
\intertext{and}
(1)_{1+k+n}=\Gamma(2+n) (2+n)_k.
\end{gather*}
Combining this with~\eqref{equ:ii} yields~\eqref{equ:B}. The proof of Theorem~\ref{Finite-sum-thm} is complete.
\end{proof}

\begin{proof}[Proof of Corollary~\ref{2F1-sum-ID}]
In~\cite[p.~66, Theorem~2.2.1]{a2000} and~\cite[Theorem~1.1]{Driver-Johnston-2006}, the Euler integral representation of the Gauss hypergeometric function ${}_2F_1$ states that, if $\Re(\tau)>\Re(\mu)>0$, then
\begin{equation}\label{Euler-Integral-Gauss-HF}
{\,}_2F_1
\begin{bmatrix}
\begin{gathered}
\lambda, \mu\\
\tau
\end{gathered}
;z 
\end{bmatrix}
=\frac{\Gamma(\tau)}{\Gamma(\mu)\Gamma(\tau-\mu)} \int_{0}^{1}t^{\mu-1}(1-t)^{\tau-\mu-1}(1-zt)^{-\lambda}\td t
\end{equation}
in the $z$ plane cut along the real axis from $1$ to $\infty$, where it is understood that $\arg t=\arg(1-t)=0$ and $(1-zt)^{-\lambda}$ has its principal value.
\par
Putting $\lambda=\frac{1}{2}$, $\mu=a$, and $\tau=1+a$ in~\eqref{Euler-Integral-Gauss-HF} leads to
\begin{equation*}
{\,}_2F_1
\begin{bmatrix}
\begin{gathered}
\tfrac{1}{2}, a\\
1+a
\end{gathered}
;z
\end{bmatrix}
=a\int_{0}^{1}t^{a-1}(1-zt)^{-1/2}\td t
=\frac{a}{z^a}\int_{0}^{z}u^{a-1}(1-u)^{-1/2}\td u
\end{equation*}
for $\Re(a)>0$ and $|z|\le1$.
Making the transform $u=\sin^2\theta$ for $\theta\in\bigl(0,\frac{\pi}{2}\bigr)$, we acquire
\begin{equation*}
{\,}_2F_1
\begin{bmatrix}
\begin{gathered}
\tfrac{1}{2}, a\\
1+a
\end{gathered}
;z
\end{bmatrix}
=\frac{2a}{z^a}\int_{0}^{\arcsin\sqrt{z}\,}\sin^{2a-1}\theta\td\theta
=\frac{2a}{z^a}\int_{0}^{\arcsin\sqrt{z}\,}\biggl(\frac{\te^{\ti\theta}-\te^{-\ti\theta}}{2\ti}\biggr)^{2a-1}\td\theta
\end{equation*}
for $\Re(a)>0$ and $|z|\le1$, where $\ti=\sqrt{-1}\,$ is the imaginary unit. Accordingly, when $2a-1\in\mathbb{N}_0$, that is, $a=\frac{1+k}{2}$ for $k\in\mathbb{N}_0$, we deduce
\begin{multline*}
\begin{aligned}
{\,}_2F_1
\begin{bmatrix}
\begin{gathered}
\tfrac{1}{2}, \tfrac{1+k}{2}\\
1+\tfrac{1+k}{2}
\end{gathered}
;z
\end{bmatrix}
&=\frac{1+k}{z^{(1+k)/2}}\int_{0}^{\arcsin\sqrt{z}\,}\biggl(\frac{\te^{\ti\theta}-\te^{-\ti\theta}}{2\ti}\biggr)^{k}\td\theta\\
&=\frac{1}{(2\ti)^k}\frac{1+k}{z^{(1+k)/2}}\sum_{j=0}^{k}(-1)^{j}\binom{k}{j} \int_{0}^{\arcsin\sqrt{z}\,} \te^{\ti(k-2j)\theta}\td\theta
\end{aligned}\\
=
\begin{dcases}
\frac{1}{(2\ti)^k}\frac{1+k}{z^{(1+k)/2}}\sum_{j=0}^{k}(-1)^{j}\binom{k}{j}\frac{\te^{\ti(k-2j)\arcsin\sqrt{z}\,}-1}{\ti(k-2j)},\quad k=1+2\ell\\
\begin{aligned}
&\frac{1+2\ell}{4^{\ell}} \frac{1}{z^{\ell+1/2}}\Biggl[\binom{2\ell}{\ell}\arcsin\sqrt{z}\,\\
&+(-1)^\ell \Biggl(\sum_{j=0}^{\ell-1}+\sum_{j=1+\ell}^{2\ell}\Biggr) (-1)^{j}\binom{2\ell}{j} \frac{\te^{\ti(2\ell-2j)\arcsin\sqrt{z}\,}-1}{\ti(2\ell-2j)}\Biggr],\quad k=2\ell
\end{aligned}
\end{dcases}
\end{multline*}
for $k,\ell\in\mathbb{N}_0$ and $|z|\le1$. From this, we can derive
\begin{equation}
\begin{aligned}\label{We-need-form}
&{\,}_2F_1
\begin{bmatrix}
\begin{gathered}
\tfrac{1}{2}, \tfrac{1}{2}+\ell\\
\tfrac{3}{2}+\ell
\end{gathered}
;z
\end{bmatrix}
=\frac{1+2\ell}{4^{\ell}} \binom{2\ell}{\ell} \frac{\arcsin\sqrt{z}\,}{z^{\ell+1/2}}\\
&\quad+\frac{(-1)^\ell}{4^{\ell}}\frac{1+2\ell}{z^{\ell+1/2}} \Biggl(\sum_{j=0}^{\ell-1}+\sum_{j=1+\ell}^{2\ell}\Biggr) (-1)^{j}\binom{2\ell}{j} \frac{\te^{\ti(2\ell-2j)\arcsin\sqrt{z}\,}-1}{\ti(2\ell-2j)}\\
&=\frac{1+2\ell}{4^{\ell}} \frac{1}{z^{\ell+1/2}} \Biggl[\binom{2\ell}{\ell} \arcsin\sqrt{z}\,\\
&\quad+(-1)^\ell\Biggl(\sum_{j=0}^{\ell-1}
+\sum_{j=1+\ell}^{2\ell}\Biggr) (-1)^{j}\binom{2\ell}{j} \frac{\sin[(2\ell-2j)\arcsin\sqrt{z}\,]}{2\ell-2j}\Biggr]
\end{aligned}
\end{equation}
and
\begin{multline}\label{we-need-formula}
{\,}_2F_1
\begin{bmatrix}
\begin{gathered}
\tfrac{1}{2}, 1+\ell\\
2+\ell
\end{gathered}
;z
\end{bmatrix}
=\frac{(-1)^{1+\ell}}{4^{\ell}}\frac{1+\ell}{z^{1+\ell}}\sum_{j=0}^{1+2\ell}(-1)^{j} \binom{1+2\ell}{j} \frac{\te^{\ti(2\ell-2j+1)\arcsin\sqrt{z}\,}-1}{2\ell-2j+1}\\
=\frac{(-1)^{1+\ell}}{4^{\ell}}\frac{1+\ell}{z^{1+\ell}}\sum_{j=0}^{1+2\ell}(-1)^{j} \binom{1+2\ell}{j} \frac{\cos[(2\ell-2j+1)\arcsin\sqrt{z}\,]-1}{2\ell-2j+1}
\end{multline}
for $\ell\in\mathbb{N}_0$ and $|z|\le1$. Consequently, we obtain the closed-form expressions~\eqref{2F1n+half-2} and~\eqref{2F1=1+n}, as well as, due to the imaginary parts in~\eqref{We-need-form} and~\eqref{we-need-formula} vanishing, the identities~\eqref{ID=01} and~\eqref{ID=02}. The proof of Corollary~\ref{2F1-sum-ID} is complete.
\end{proof}

\begin{proof}[Proof of Corollary~\ref{(-1)series-cor}]
Substituting the formulas~\eqref{2F1n+half-2} and~\eqref{2F1=1+n} into~\eqref{equ:ii} respectively produces~\eqref{closed=int} and~\eqref{closed=integer}.
The proof of Corollary~\ref{(-1)series-cor} is complete.
\end{proof}

\begin{proof}[Proof of Corollary~\ref{3F2-pm1-cor}]
Combining~\eqref{equ:ii} and~\eqref{equ:i}, we deduce that
\begin{equation}\label{3F2=2F1}
{\,}_3F_2
\begin{bmatrix}
\begin{gathered}
 1,\tfrac{3}{2}, 1+\alpha\\
2, 2+\alpha
\end{gathered}
;z 
\end{bmatrix}
=\begin{dcases}
\frac{2(1+\alpha)}{\alpha}\frac{1}{z}\biggl({\,}_2F_1
\begin{bmatrix}
\begin{gathered}
\tfrac{1}{2}, \alpha\\
 1+\alpha
\end{gathered}
;z 
\end{bmatrix}
-1\biggr), & \alpha\ne0\\
\frac{4}{z}\ln\frac{2}{1+\sqrt{1-z}\,}, & \alpha=0
\end{dcases}
\end{equation}
for $\alpha\in\mathbb{C}\setminus\mathbb{Z}^-$ and $|z|\le1$. Taking $\alpha=\frac{1}{2}+n$ and $\alpha=1+n$ in~\eqref{3F2=2F1} and making use of~\eqref{2F1n+half-2} and~\eqref{2F1=1+n}, we arrive at
\begin{align*}
{\,}_3F_2
\begin{bmatrix}
\begin{gathered}
 1,\tfrac{3}{2}, \tfrac{3}{2}+n\\
2, \tfrac{5}{2}+n
\end{gathered}
;z 
\end{bmatrix}
&=
\frac{3+2n}{1+2n}\frac{2}{z}\biggl({\,}_2F_1
\begin{bmatrix}
\begin{gathered}
\tfrac{1}{2}, \tfrac{1}{2}+n\\
 \tfrac{3}{2}+n
\end{gathered}
;z 
\end{bmatrix}
-1\biggr)\\
&=\frac{3+2n}{1+2n}\frac{2}{z}\Biggl(\frac{1+2n}{4^{n}} \frac{1}{z^{n+1/2}}\Biggl[\binom{2n}{n} \arcsin\sqrt{z}\,\\
+(-1)^n\Biggl(\sum_{j=0}^{n-1}&+\sum_{j=1+n}^{2n}\Biggr) (-1)^{j}\binom{2n}{j} \frac{\sin[2(n-j)\arcsin\sqrt{z}\,]}{2(n-j)}\Biggr]-1\Biggr)
\end{align*}
and
\begin{align*}
{\,}_3F_2
\begin{bmatrix}
\begin{gathered}
 1,\tfrac{3}{2}, 2+n\\
2, 3+n
\end{gathered}
;z 
\end{bmatrix}
&=
\frac{2+n}{1+n}\frac{2}{z}\biggl({\,}_2F_1
\begin{bmatrix}
\begin{gathered}
\tfrac{1}{2}, 1+n\\
 2+n
\end{gathered}
;z 
\end{bmatrix}
-1\biggr)\\
&=
\frac{2+n}{1+n}\frac{2}{z}\Biggl[\frac{(-1)^{1+n}}{4^{n}}\frac{1+n}{z^{1+n}}\\
&\quad\times\sum_{j=0}^{1+2n}(-1)^{j} \binom{1+2n}{j} \frac{\cos[(2n-2j+1)\arcsin\sqrt{z}\,]-1}{2n-2j+1}-1\Biggr]
\end{align*}
for $n\in\mathbb{N}_0$. As a result, the closed-form expressions~\eqref{3F2-closed1} and~\eqref{3F2-closed2} are derived.
The proof of Corollary~\ref{3F2-pm1-cor} is thus complete.
\end{proof}

\begin{proof}[Proof of Corollary~\ref{Beta-Closed-Cor}]
In~\cite[p.~455, Entry~28]{p1990}, we find
\begin{equation}\label{Entry28NIST-HB-2010}
{\,}_2F_1
\begin{bmatrix}
\begin{gathered}
a, b\\
1+b
\end{gathered}
;z 
\end{bmatrix}
=\frac{b}{z^b}B_z(b,1-a)
\end{equation}
for $\Re(b)>0$ and $0\le z\le1$. This means that
\begin{equation}\label{2F1-Beta}
{\,}_2F_1
\begin{bmatrix}
\begin{gathered}
\tfrac{1}{2}, \alpha\\
 1+\alpha
\end{gathered}
;z 
\end{bmatrix}
=\frac{\alpha}{z^\alpha}B_z\biggl(\alpha, \frac12\biggr)
\end{equation}
for $\Re(\alpha)>0$ and $0\le z\le1$. According to~\eqref{2F1-Beta}, we have
\begin{equation}\label{2F1-Beta1}
{\,}_2F_1
\begin{bmatrix}
\begin{gathered}
\tfrac{1}{2}, \tfrac{1}{2}+n\\
 \tfrac{3}{2}+n
\end{gathered}
;z 
\end{bmatrix}
=\frac{1+2n}{2z^{n+1/2}}B_z\biggl(\frac{1}{2}+n,\frac12\biggr)
\end{equation}
and
\begin{equation}\label{2F1-Beta2}
{\,}_2F_1
\begin{bmatrix}
\begin{gathered}
\tfrac{1}{2}, 1+n\\
2+n
\end{gathered}
;z 
\end{bmatrix}
=\frac{1+n}{z^{1+n}}B_z\biggl(1+n,\frac12\biggr)
\end{equation}
for $n\in\mathbb{N}_0$ and $0\le z\le1$.
Comparing~\eqref{2F1-Beta1} and~\eqref{2F1-Beta2} with~\eqref{2F1n+half-2} and~\eqref{2F1=1+n} leads to the closed-form expressions~\eqref{beta-closed1} and~\eqref{beta-closed2}. The proof of Corollary~\ref{Beta-Closed-Cor} is complete.
\end{proof}

\begin{proof}[Proof of Corollary~\ref{derv-form-cor}]
In~\cite[p.~109, Example~5.1]{Temme-96-book}, we find
\begin{equation}\label{p109Example5.1Temme-96-book}
{\,}_2F_1
\begin{bmatrix}
\begin{gathered}
\tfrac{1}{2}, \tfrac12\\
\tfrac{3}{2}
\end{gathered}
;-z^2
\end{bmatrix}
=\frac{\ln\bigl(z+\sqrt{1+z^2}\,\bigr)}{z}, \quad |z|\le1.
\end{equation}
The formula~\eqref{p109Example5.1Temme-96-book} can be rewritten as
\begin{equation}\label{p109Example5.1}
{\,}_2F_1
\begin{bmatrix}
\begin{gathered}
\tfrac{1}{2}, \tfrac12\\
\tfrac{3}{2}
\end{gathered}
;z
\end{bmatrix}
=\frac{\ln\bigl(\sqrt{1-z}\,+\sqrt{-z}\,\bigr)}{\sqrt{-z}\,}
=\frac{\arcsin\sqrt{z}\,}{\sqrt{z}\,}, \quad |z|\le1.
\end{equation}
In~\cite[p.~388, Entry~15.5.7]{NIST-HB-2010}, we find
\begin{multline*}
\biggl[(1-z)\frac{\td}{\td z}(1-z)\biggr]^n \biggl((1-z)^{a-1}{\,}_2F_1
\begin{bmatrix}
\begin{gathered}
a, b\\
c
\end{gathered}
;z
\end{bmatrix}\biggr)\\
=(-1)^n\frac{(a)_n(c-b)_n}{(c)_n}(1-z)^{a+n-1}{\,}_2F_1
\begin{bmatrix}
\begin{gathered}
a+n, b\\
c+n
\end{gathered}
;z
\end{bmatrix}.
\end{multline*}
Taking $(a,b;c)=\bigl(\frac{1}{2},\frac{1}{2};\frac{3}{2}\bigr)$ yields
\begin{multline*}
\biggl[(1-z)\frac{\td}{\td z}(1-z)\biggr]^n \biggl(\frac{1}{\sqrt{1-z}\,}{\,}_2F_1
\begin{bmatrix}
\begin{gathered}
\tfrac{1}{2}, \tfrac{1}{2}\\
\tfrac{3}{2}
\end{gathered}
;z
\end{bmatrix}\biggr)\\
=(-1)^n\frac{\bigl(\frac{1}{2}\bigr)_n(1)_n}{\bigl(\frac{3}{2}\bigr)_n}(1-z)^{n-1/2}{\,}_2F_1
\begin{bmatrix}
\begin{gathered}
\tfrac{1}{2}, \tfrac{1}{2}+n\\
\tfrac{3}{2}+n
\end{gathered}
;z
\end{bmatrix}.
\end{multline*}
Accordingly, using the formula~\eqref{p109Example5.1}, we derive
\begin{equation}\label{1+n2+n}
\begin{aligned}
{\,}_2F_1
\begin{bmatrix}
\begin{gathered}
\tfrac{1}{2}, \tfrac{1}{2}+n\\
\tfrac{3}{2}+n
\end{gathered}
;z
\end{bmatrix}
&=\frac{(-1)^n}{n!} \frac{1+2n}{(1-z)^{n-1/2}}
\biggl[(1-z)\frac{\td}{\td z}(1-z)\biggr]^n \frac{{\,}_2F_1
\begin{bmatrix}
\begin{gathered}
\tfrac{1}{2}, \tfrac{1}{2}\\
\tfrac{3}{2}
\end{gathered}
;z
\end{bmatrix}}{\sqrt{1-z}\,}\\
&=\frac{(-1)^n}{n!} \frac{1+2n}{(1-z)^{n-1/2}}
\biggl[(1-z)\frac{\td}{\td z}(1-z)\biggr]^n \frac{\arcsin\sqrt{z}\,}{\sqrt{z(1-z)}\,}
\end{aligned}
\end{equation}
for $n\in\mathbb{N}_0$. Comparing~\eqref{1+n2+n} with~\eqref{2F1n+half-2} and simplifying result in~\eqref{deriv-form}. The proof of Corollary~\ref{derv-form-cor} is complete.
\end{proof}

\section{Remarks}
In this section, we state several remarks, mainly discussing the evaluations of the Gauss hypergeometric function ${}_2F_1$
for $\alpha\in\mathbb{C}\setminus\mathbb{Z}^-$ and $|z|\le1$.

\begin{rem}\label{rem1}
Taking $z=1$ in~\eqref{equ:ii} in Theorem~\ref{Arjun-Gauss-thm}, we acquire
\begin{equation}\label{equ:ii-sum}
\sum_{k=1}^{\infty}\binom{2k}{k}\frac{1}{\alpha+k}\frac{1}{4^k}
=\begin{dcases}\frac{1}{\alpha}\biggl({\,}_2F_1
\begin{bmatrix}
\begin{gathered}
\tfrac{1}{2}, \alpha\\
 1+\alpha
\end{gathered}
;1
\end{bmatrix}
-1\biggr), & \alpha\ne0\\
2\ln2, & \alpha=0
\end{dcases}
\end{equation}
for $\alpha\in\mathbb{C}\setminus\mathbb{Z}^-$.
The Gauss hypergeometric summation theorem collected in~\cite[p.~111, Eq.~(5.6)]{Temme-96-book} reads that
\begin{equation}\label{hyperg-summ-thm}
{\,}_2F_1
\begin{bmatrix}
\begin{gathered}
 a, b\\
c
\end{gathered}
;1 
\end{bmatrix}
=\frac{\Gamma(c)\Gamma(c-a-b)}{\Gamma(c-a)\Gamma(c-b)} 
\end{equation}
provided $\Re(c-a-b)>0$. Taking $a=\frac{1}{2}$, $b=\alpha$, and $c=1+\alpha$ for $\alpha\ne0$ in~\eqref{hyperg-summ-thm}, we acquire
\begin{equation*}
{\,}_2F_1
\begin{bmatrix}
\begin{gathered}
\tfrac{1}{2}, \alpha\\
1+\alpha
\end{gathered}
;1 
\end{bmatrix}
=\frac{\Gamma(1+\alpha)\Gamma\bigl(\frac{1}{2}\bigr)}{\Gamma\bigl(\frac{1}{2}+\alpha\bigr)\Gamma(1)}
=\frac{\sqrt{\pi}\,\alpha\Gamma(\alpha)}{\Gamma\bigl(\frac{1}{2}+\alpha\bigr)}, \quad \alpha\in\mathbb{C}\setminus\{0,-1,-2,\dotsc\}.
\end{equation*}
Substituting this into~\eqref{equ:ii-sum} leads to
\begin{equation}\label{equ:vVV}
\sum_{k=1}^{\infty}\binom{2k}{k}\frac{1}{\alpha+k}\frac{1}{4^k}
=
\begin{dcases}
\frac{\sqrt{\pi}\,\Gamma(\alpha)}{\Gamma\bigl(\frac{1}{2}+\alpha\bigr)}-\frac{1}{\alpha}, & \alpha\ne0\\
2\ln2, & \alpha=0
\end{dcases}
\end{equation}
for $\alpha\in\mathbb{C}\setminus\mathbb{Z}^-$.
\par
Setting $z=1$ in~\eqref{equ:i} yields
\begin{equation}\label{z=1}
\sum_{k=1}^{\infty}\binom{2k}{k}\frac{1}{\alpha+k}\frac{1}{4^k} =\frac{1}{2(1+\alpha)}{\,}_3F_2
\begin{bmatrix}
\begin{gathered}
 1,\tfrac{3}{2}, 1+\alpha\\
2, 2+\alpha
\end{gathered}
;1 
\end{bmatrix}.
\end{equation}
In~\cite[p.~536, Entry~29]{p1990}, we find the formula
\begin{equation}\label{p1990-pages}
{\,}_3F_2
\begin{bmatrix}
\begin{gathered}
1, a, b\\
2, c
\end{gathered}
;1 
\end{bmatrix} =\frac{c-1}{(a-1)(b-1)}\biggl[\frac{\Gamma(c-1)\Gamma(c-a-b+1)}{\Gamma(c-a)\Gamma(c-b)}-1\biggr]
\end{equation}
for $a,b,c\ne1$ and $\Re(c-a-b)>-1$.
Letting $\alpha\ne0$, $a =\frac{3}{2}$, $b=1+\alpha$, and $c=2+\alpha$ in~\eqref{p1990-pages} results in
\begin{equation*}
{\,}_3F_2
\begin{bmatrix}
\begin{gathered}
1, \tfrac{3}{2}, 1+\alpha\\
2, 2+\alpha
\end{gathered}
;1 
\end{bmatrix}
=\frac{2(1+\alpha)}{\alpha}\biggl[\frac{\Gamma(1+\alpha)\Gamma\bigl(\frac{1}{2}\bigr)}{\Gamma\bigl(\frac{1}{2}+\alpha\bigr)\Gamma(1)}-1\biggr]
=\frac{2(1+\alpha)}{\alpha}\biggl[\frac{\sqrt{\pi}\,\alpha\Gamma(\alpha)}{\Gamma\bigl(\frac{1}{2}+\alpha\bigr)}-1\biggr].
\end{equation*}
Substituting this into~\eqref{z=1} results in the formula~\eqref{equ:vVV} for the case $\alpha\ne0$ once again. Thus, the claim~\eqref{anonymous1} is confirmed.
\par
Letting $z\to1^-$ in~\eqref{equ:B} in Theorem~\ref{Finite-sum-thm} and employing~\eqref{equ:vVV}, we can confirm the claim~\eqref{anonymous2} immediately.
\end{rem}

\begin{rem}
Similar to the closed forms~\eqref{2F1n+half-2}, \eqref{2F1=1+n}, \eqref{3F2-closed1}, and~\eqref{3F2-closed2}, there have been several closed forms of the Gauss hypergeometric functions
\begin{gather*}
{\,}_2F_1
\begin{bmatrix}
\begin{gathered}
\tfrac{1}{2}+n,1+n\\
\tfrac{3}{2}+n
\end{gathered}
;-z^2
\end{bmatrix}, \quad
{\,}_2F_1
\begin{bmatrix}
\begin{gathered}
\tfrac{1}{2}+n,\tfrac{1}{2}+n\\
\tfrac{3}{2}+n
\end{gathered}
;-z^2
\end{bmatrix},\quad
{\,}_2F_1
\begin{bmatrix}
\begin{gathered}
1,\tfrac{1}{2}+n\\
\tfrac{3}{2}+n
\end{gathered}
;z^2
\end{bmatrix},\\
{\,}_2F_1
\begin{bmatrix}
\begin{gathered}
\tfrac{1}{2}-\tfrac{n}{2},1-\tfrac{n}{2}\\
\tfrac{3}{2}-n
\end{gathered}
;z^2
\end{bmatrix},\quad
{\,}_2F_1
\begin{bmatrix}
\begin{gathered}
-\tfrac{n}{2},\tfrac{1}{2}-\tfrac{n}{2}\\
\tfrac{1}{2}-n
\end{gathered}
;z^2
\end{bmatrix}.
\end{gather*}
\par
In~\cite[Corollary~4.1]{Qi-Wilf.tex}, Qi established the closed form
\begin{multline}\label{Gauss-HF-Spec-Value}
{\,}_2F_1
\begin{bmatrix}
\begin{gathered}
\tfrac{1}{2}+n,1+n\\
\tfrac{3}{2}+n
\end{gathered}
;-1
\end{bmatrix}\\
=\frac{(1+2n)!!}{(2n)!!}\frac{\pi}{4}
+\frac{1+2n}{2^{2n}}\sum_{k=1}^{n} (-1)^{k} \binom{2n-k}{n} \frac{2^{k/2}}{k}\sin\frac{3k\pi}{4}, \quad n\in\mathbb{N}_0.
\end{multline}
In~\cite[Theorem~3]{axioms-2962911.tex}, Li and Qi generalized the closed form~\eqref{Gauss-HF-Spec-Value} as
\begin{equation}\label{T1.1Vid=unas-2003-deriv}
{\,}_2F_1
\begin{bmatrix}
\begin{gathered}
\tfrac{1}{2}+n,1+n\\
\tfrac{3}{2}+n
\end{gathered}
;-z^2
\end{bmatrix}
=
\begin{dcases}
\frac{(1+2n)!!}{(2n)!!}\frac{1}{z^{2n}}\biggl[\frac{\arctan z}{z} -\frac{\mathcal{Q}_{n-1}\bigl(z^2\bigr)}{(1+z^2)^n}\biggr], & z\ne0\\
1, & z=0
\end{dcases}
\end{equation}
for $|z|<1$ and $n\in\mathbb{N}_0$, where $\mathcal{Q}_{-1}(z)=0$ and
\begin{equation}\label{mathcal(Q)-Eq}
\mathcal{Q}_n(z)=\sum_{k=0}^{n} \Biggl[\sum_{j=0}^{k}\frac{(-1)^j}{1+2j} \binom{1+n}{k-j}\Biggr]z^k,\quad n\in\mathbb{N}_0.
\end{equation}
\par
In~\cite[Theorem~1.3]{Gauss-Milovanovic-Qi.tex}, Milovanovi\'c and Qi presented
\begin{multline}\label{Milovanovic-claim-eq}
{\,}_2F_1
\begin{bmatrix}
\begin{gathered}
\tfrac{1}{2}+n,\tfrac{1}{2}+n\\
\tfrac{3}{2}+n
\end{gathered}
;-z^2
\end{bmatrix}\\
=
\begin{dcases}
\frac{1+2n}{z^{2n}}\biggl[\frac{\arcsinh z}{z}-\frac{1}{(2n-1)!!}\frac{R_{n-1}\bigl(z^2\bigr)}{(1+z^2)^{n-1/2}}\biggr], & z\ne0\\
1, & z=0
\end{dcases}
\end{multline}
and
\begin{equation}\label{Qi-claim-eq}
{\,}_2F_1
\begin{bmatrix}
\begin{gathered}
1,\tfrac{1}{2}+n\\
\tfrac{3}{2}+n
\end{gathered}
;z^2
\end{bmatrix}
=
\begin{dcases}
\frac{1+2n}{z^{2n}} \Biggl[\frac{\arctanh z}{z}-\frac{\mathcal{R}_{n-1}\bigl(z^2\bigr)}{(2n-1)!!}\Biggr], & z\ne0\\
1, & z=0
\end{dcases}
\end{equation}
for $n\in\mathbb{N}_0$ and $|z|<1$, where
\begin{gather}\label{Rn(z)-1}
R_{-1}(z)=0,\quad \mathcal{R}_{-1}(z)=0,\\
R_n(z)=(1+2n)!!\sum_{k=0}^{n}\Biggl[\sum_{j=0}^{k}\frac{1}{1+2j}\binom{n-j}{k-j}\Biggr]z^k,\label{Rn(z)}
\end{gather}
and
\begin{equation}\label{CalRn(z)}
\mathcal{R}_{n}(z)=(1+2n)!!\sum_{j=0}^{n}\frac{z^{j}}{1+2j}
\end{equation}
for $n\in\mathbb{N}_0$ are both positive integer polynomials in $z$ of degree $n\in\mathbb{N}_0$.
\par
In~\cite{Qi-Izan-Peraz-Gauss.tex}, Qi and his coauthors acquired
\begin{multline}\label{2F1(izan-peraz)}
{\,}_2F_1
\begin{bmatrix}
\begin{gathered}
\tfrac{1}{2}-\tfrac{n}{2},1-\tfrac{n}{2}\\
\tfrac{3}{2}-n
\end{gathered}
;z^2
\end{bmatrix}
=\frac{1}{4\binom{2n-2}{n-1}}\sum_{k=1}^{n}2^k\binom{2n-k-1}{n-1}\binom{n-1}{k-1} \\
\times\bigl[(1-z)^{n-k}+(-1)^{k-1}(1+z)^{n-k}\bigr]z^{k-1}
\end{multline}
and
\begin{multline}\label{2F1(izan-peraz)-more}
{\,}_2F_1
\begin{bmatrix}
\begin{gathered}
-\tfrac{n}{2},\tfrac{1}{2}-\tfrac{n}{2}\\
\tfrac{1}{2}-n
\end{gathered}
;z^2
\end{bmatrix}
=\frac{n!}{2\binom{2n}{n}}\sum_{k=0}^{n} \frac{2^{k}}{k!}\binom{2n-2k}{n-k}\\
\times\sum_{j=0}^{k}\frac{(2k-2j-1)!!}{(n-j)!}\binom{2k-j-1}{j-1} \bigl[(1-z)^{n-j}+(-1)^{j}(1+z)^{n-j}\bigr]z^{j}
\end{multline}
for $n\in\mathbb{N}$ and $|z|<1$.
Importantly, making use of these two closed forms, they discovered the combinatorial identity
\begin{equation}\label{ID-n-Factorial2sums}
\sum_{k=0}^{n} \frac{2^{k}}{k!}\binom{2n-2k}{n-k} \sum_{j=0}^{k}\frac{(-1)^{j}}{2^j}\frac{(2k-2j-1)!!}{(n-j)!}\binom{2k-j-1}{j-1}
=\frac{1}{n!}, \quad n\in\mathbb{N}_0,
\end{equation}
with the understanding of $\binom{-1}{-1}=1$.
\par
In~\cite{Gauss-2Formula.tex}, the combinatorial identity~\eqref{ID-n-Factorial2sums} and the formulas~\eqref{2F1(izan-peraz)} and~\eqref{2F1(izan-peraz)-more} were further generalized. Moreover, Fr\'etigny and Qi studied in~\cite{AADM-3316.tex} higher-order derivatives of two arctangent-based functions and their relation to the Gauss hypergeometric functions treated in~\cite{axioms-2962911.tex, Gauss-Milovanovic-Qi.tex, Qi-Wilf.tex}.
\end{rem}

\begin{rem}
Applying the Pfaff transformation
\begin{equation}\label{p559Entry15.3.5abram}
{\,}_2F_1
\begin{bmatrix}
\begin{gathered}
a,b\\
c
\end{gathered}
;z
\end{bmatrix}
=\frac1{(1-z)^b}
{\,}_2F_1
\begin{bmatrix}
\begin{gathered}
c-a,b\\
c
\end{gathered}
;\dfrac{z}{z-1}
\end{bmatrix}
\end{equation}
in~\cite[p.~559, Entry~15.3.5]{abram} and employing the closed form~\eqref{T1.1Vid=unas-2003-deriv}, we acquire
\begin{equation*}
{\,}_2F_1
\begin{bmatrix}
\begin{gathered}
1,1+n\\
\tfrac{3}{2}+n
\end{gathered}
;\dfrac{z^2}{z^2+1}
\end{bmatrix}
=\frac{(1+2n)!!}{(2n)!!}\frac{(1+z^2)^{1+n}}{z^{2n}}\biggl[\frac{\arctan z}{z} -\frac{\mathcal{Q}_{n-1}\bigl(z^2\bigr)}{(1+z^2)^n}\biggr],
\end{equation*}
that is,
\begin{multline}\label{final-2F1}
{\,}_2F_1
\begin{bmatrix}
\begin{gathered}
1,1+n\\
\tfrac{3}{2}+n
\end{gathered}
;z
\end{bmatrix}\\*
=\frac{(1+2n)!!}{(2n)!!}\frac{1}{z^{n}(1-z)}\Biggl[\frac{\arctan\sqrt{\frac{z}{1-z}}\,}{\sqrt{\frac{z}{1-z}}\,} -(1-z)^{n}\mathcal{Q}_{n-1}\biggl(\frac{z}{1-z}\biggr)\Biggr],
\end{multline}
for $|z|<1$ and $n\in\mathbb{N}_0$, where $\mathcal{Q}_{n}$ is defined by~\eqref{mathcal(Q)-Eq}.
The substitution $w=\frac{z^2}{1+z^2}$ only covers $0\le w<\frac12$ (the range of
$|z|<1$ in~\eqref{T1.1Vid=unas-2003-deriv}), but both sides of~\eqref{final-2F1}
are analytic in the whole disk $|z|<1$. Hence~\eqref{final-2F1} holds there by analytic continuation.
\end{rem}

\begin{rem}
Making use of the transform~\eqref{p559Entry15.3.5abram} and the closed form~\eqref{T1.1Vid=unas-2003-deriv}, we derive
\begin{equation*}
{\,}_2F_1
\begin{bmatrix}
\begin{gathered}
\tfrac{1}{2}+n,\tfrac{1}{2}+n\\
\tfrac{3}{2}+n
\end{gathered}
;-z^2
\end{bmatrix}
=\frac1{(1+z^2)^{n+1/2}}
{\,}_2F_1
\begin{bmatrix}
\begin{gathered}
1,\tfrac{1}{2}+n\\
\tfrac{3}{2}+n
\end{gathered}
;\dfrac{z^2}{1+z^2}
\end{bmatrix}, \quad n\in\mathbb{N}_0.
\end{equation*}
This connects Theorems~1.1 and~1.2 in~\cite{Gauss-Milovanovic-Qi.tex}, in other words, this shows that Theorems~1.1 and~1.2 in~\cite{Gauss-Milovanovic-Qi.tex} are equivalent to each other.
\par
From~\eqref{Milovanovic-claim-eq} and~\eqref{Qi-claim-eq}, we acquire an identity
\begin{equation}\label{arcsin-arctanh-id}
\frac{\arcsinh z-\arctanh\frac{z}{\sqrt{1+z^2}\,}}{z}
=\frac{1}{(2n-1)!!}\biggl[\frac{R_{n-1}\bigl(z^2\bigr)}{(1+z^2)^{n-1/2}} -\frac{\mathcal{R}_{n-1}\bigl(\frac{z^2}{1+z^2}\bigr)}{\sqrt{1+z^2}\,}\biggr]
\end{equation}
for $n\in\mathbb{N}_0$ and $|z|<1$,
where $R_{n}$ and $\mathcal{R}_{n}$ are defined by~\eqref{Rn(z)-1}, \eqref{Rn(z)}, and~\eqref{CalRn(z)}.
We notice that the left-hand side of the identity~\eqref{arcsin-arctanh-id} is independent of $n\in\mathbb{N}_0$, but the right-hand side depends on $n\in\mathbb{N}_0$. Taking $n=0$ in~\eqref{arcsin-arctanh-id} yields
\begin{equation}\label{8ed-p60}
\arcsinh z=\arctanh\frac{z}{\sqrt{1+z^2}}, \quad |z|<1
\end{equation}
and then
\begin{equation}\label{p.16-last-line}
\frac{R_{n-1}\bigl(z^2\bigr)}{(1+z^2)^{n-1}}=\mathcal{R}_{n-1}\biggl(\frac{z^2}{1+z^2}\biggr), \quad n\in\mathbb{N}_0.
\end{equation}
The identity~\eqref{8ed-p60} is listed in~\cite[p.~60]{Gradshteyn-Ryzhik-Table-8th} and the equality~\eqref{p.16-last-line} appeared in~\cite[p.~16]{Gauss-Milovanovic-Qi.tex}.
\end{rem}

\begin{rem}
Utilizing the transform~\eqref{p559Entry15.3.5abram} again and employing the closed forms~\eqref{2F1n+half-2} and~\eqref{2F1=1+n} in Corollary~\ref{2F1-sum-ID} yield the closed forms
\begin{multline}\label{2F1arccos}
{\,}_2F_1
\begin{bmatrix}
\begin{gathered}
\tfrac{1}{2}+n,1+n\\
\tfrac{3}{2}+n
\end{gathered}
;z
\end{bmatrix}
=\frac{1+2n}{4^{n}(-z)^{n+1/2}}\Biggl[\binom{2n}{n} \arcsin\sqrt{\frac{z}{z-1}}\,\\
+(-1)^n\Biggl(\sum_{k=0}^{n-1} +\sum_{k=1+n}^{2n}\Biggr) (-1)^{k}\binom{2n}{k} \frac{\sin\Bigl[2(n-k)\arcsin\sqrt{\frac{z}{z-1}}\,\Bigr]}{2(n-k)}\Biggr]
\end{multline}
and
\begin{multline}\label{2F1arcsin}
{\,}_2F_1
\begin{bmatrix}
\begin{gathered}
1+n,\tfrac{3}{2}+n\\
2+n
\end{gathered}
;z
\end{bmatrix}
\\
=\frac{1+n}{4^{n}z^{1+n}}\sum_{k=0}^{1+2n}(-1)^{k} \binom{1+2n}{k} \frac{\cos\Bigl[(2n-2k+1)\arcsin\sqrt{\frac{z}{z-1}}\,\Bigr]-1}{2n-2k+1}
\end{multline}
for $n\in\mathbb{N}_0$ and $0<|z|<1$, where principal branches are used.
\end{rem}

\begin{rem}
Applying the transform~\eqref{p559Entry15.3.5abram} to the closed forms~\eqref{2F1(izan-peraz)} and~\eqref{2F1(izan-peraz)-more} results in the closed forms
\begin{multline}\label{2F1izan-peraz}
{\,}_2F_1
\begin{bmatrix}
\begin{gathered}
1-\tfrac{n}{2},1-\tfrac{n}{2}\\
\tfrac{3}{2}-n
\end{gathered}
;z
\end{bmatrix}
=\frac{(1-z)^{n/2-1}}{4\binom{2n-2}{n-1}} \sum_{k=1}^{n}2^k\binom{2n-k-1}{n-1}\binom{n-1}{k-1} \\
\times\biggl[\biggl(1-\sqrt{\frac{z}{z-1}}\,\biggr)^{n-k}+(-1)^{k-1}\biggl(1+\sqrt{\frac{z}{z-1}}\,\biggr)^{n-k}\biggr] \biggl(\frac{z}{z-1}\biggr)^{(k-1)/2}
\end{multline}
and
\begin{multline}\label{2F1(izan-peraz)-m}
{\,}_2F_1
\begin{bmatrix}
\begin{gathered}
\tfrac{1}{2}-\tfrac{n}{2},\tfrac{1}{2}-\tfrac{n}{2}\\
\tfrac{1}{2}-n
\end{gathered}
;z
\end{bmatrix}
=\frac{n!(1-z)^{(n-1)/2}}{2\binom{2n}{n}}\sum_{k=0}^{n} \frac{2^{k}}{k!}\binom{2n-2k}{n-k} \sum_{j=0}^{k}\frac{(2k-2j-1)!!}{(n-j)!} \\
\times\binom{2k-j-1}{j-1} \biggl[\biggl(1-\sqrt{\frac{z}{z-1}}\,\biggr)^{n-j}+(-1)^{j}\biggl(1+\sqrt{\frac{z}{z-1}}\,\biggr)^{n-j}\biggr] \biggl(\frac{z}{z-1}\biggr)^{j/2}
\end{multline}
for $n\in\mathbb{N}$ and $|z|<1$. Since both ${}_2F_1$ functions on the left-hand sides are analytic on $\mathbb{C}\setminus[1,\infty)$ with principal branches, the transformation is valid beyond the region $\bigl|\frac{z}{z-1}\bigr|<1$ where the original series converges.
\end{rem}

\begin{rem}
In~\cite[Theorem~1]{Gauss-hyperg-Int.tex}, employing three methods without the Euler integral representation~\eqref{Euler-Integral-Gauss-HF} of the Gauss hypergeometric function ${\,}_2F_1$, Qi presented the integral representation
\begin{equation}\label{Qi-hypergeometric-integral}
{\,}_2F_1
\begin{bmatrix}
\begin{gathered}
a-\tfrac{1}{2},a\\
a+\tfrac{1}{2}
\end{gathered}
;z 
\end{bmatrix}
=(2a-1)\int_{0}^{1}\frac{t^{2(a-1)}}{(1-zt^2)^{a}}\td t
\end{equation}
for $a\in\bigl(\frac{1}{2},\infty\bigr)$ and $z\in\mathbb{C}\setminus[1,\infty)$, and gave seven applications of~\eqref{Qi-hypergeometric-integral}.
\end{rem}

\begin{rem}
Concretely, the infinite series~\eqref{Qi-infi-ser} of the case $(-1)^k$ for $m=1, 2, 3$ can be computed directly as follows.
\par
Letting $\alpha=1$ and $z\to(-1)^+$ in the formula~\eqref{equ:ii}, we arrive at
\begin{equation*}
\sum_{k=1}^{\infty}\binom{2k}{k}\frac{1}{1+k}\frac{(-1)^k}{4^k}
={\,}_2F_1
\begin{bmatrix}
\begin{gathered}
\tfrac{1}{2}, 1\\
 2
\end{gathered}
;- 1 
\end{bmatrix}-1.
\end{equation*}
In~\cite[p.~473, Entry~84]{p1990}, we find
\begin{equation*}
{\,}_2F_1
\begin{bmatrix}
\begin{gathered}
\tfrac{1}{2}, 1\\
 2
\end{gathered}
;z 
\end{bmatrix} =\frac{2}{1 +\sqrt{1-z}\,}.
\end{equation*}
Accordingly, we arrive at
\begin{equation*}
\sum_{k=1}^{\infty}\binom{2k}{k}\frac{1}{1+k}\frac{(-1)^k}{4^k}
=\frac{2}{1 +\sqrt{2}\,}-1
=2\sqrt{2}\,-3.
\end{equation*}
\par
For $\alpha=2$ and $z\to(-1)^+$ in the formula~\eqref{equ:ii}, we gain
\begin{equation*}
\sum_{k=1}^{\infty}\binom{2k}{k}\frac{1}{2+k}\frac{(-1)^k}{4^k}
=\frac{1}{2}\biggl({\,}_2F_1
\begin{bmatrix}
\begin{gathered}
\tfrac{1}{2}, 2\\
 3
\end{gathered}
;- 1 
\end{bmatrix}-1\biggr).
\end{equation*} 
Using the formula
\begin{equation*}
{\,}_2F_1
\begin{bmatrix}
\begin{gathered}
\tfrac{1}{2}, 2\\
 3
\end{gathered}
;z 
\end{bmatrix} =\frac{4}{3z^2}\bigl[2-( 2+z)\sqrt{1-z}\,\bigr]
\end{equation*}
in~\cite[p.~474, Entry~100]{p1990}, we obtain
\begin{equation*}
\sum_{k=1}^{\infty}\binom{2k}{k}\frac{1}{2+k}\frac{(-1)^k}{4^k}
=\frac{1}{2}\biggl[\frac{4}{3}\bigl(2-\sqrt{2}\,\bigr)-1\biggr]
=-\frac{4\sqrt{2}\,-5}{6}.
\end{equation*}
\par
For $\alpha=3$ and $z\to(-1)^+$ in the formula~\eqref{equ:ii}, we acquire
\begin{equation*}
\sum_{k=1}^{\infty}\binom{2k}{k}\frac{1}{3+k}\frac{(-1)^k}{4^k}
=\frac{1}{3}\biggl({\,}_2F_1
\begin{bmatrix}
\begin{gathered}
\tfrac{1}{2}, 3\\
 4
\end{gathered}
;- 1 
\end{bmatrix}-1\biggr).
\end{equation*} 
Employing the formula
\begin{equation*}
{\,}_2F_1
\begin{bmatrix}
\begin{gathered}
\tfrac{1}{2}, 3\\
 4
\end{gathered}
;z 
\end{bmatrix} =\frac{2}{5z^3}\bigl[8-\bigl(8+4z+3z^2\bigr)\sqrt{1-z}\,\bigr]
\end{equation*}
in~\cite[p.~474, Entry~116]{p1990}, we acquire
\begin{equation*}
\sum_{k=1}^{\infty}\binom{2k}{k}\frac{1}{3+k}\frac{(-1)^k}{4^k}
=\frac{1}{3}\biggl[-\frac{2}{5}\bigl(8-7\sqrt{2}\,\bigr)-1\biggr]
=-\frac{7\bigl(3-2\sqrt{2}\,\bigr)}{15}.
\end{equation*}
\end{rem}

\begin{rem}
Taking $\alpha=1$ in Theorems~\ref{Arjun-Gauss-thm} and~\ref{Finite-sum-thm} and
adding the $k=0$ term $\binom{0}{0}\frac{1}{1+0}\bigl(\frac z4\bigr)^0=1$ to the
sums, which start at $k=1$, we obtain
\begin{gather*}
\sum_{k=0}^{\infty}\binom{2k}{k}\frac{1}{1+k}\biggl(\frac{z}{4}\biggr)^k
=
{\,}_2F_1
\begin{bmatrix}
\begin{gathered}
\tfrac{1}{2}, 1\\
2
\end{gathered}
;z 
\end{bmatrix},\\
\sum_{k=0}^{\infty}\binom{2k+2}{1+k}\frac{1}{2+k}\biggl(\frac{z}{4}\biggr)^{k} 
={\,}_3F_2
\begin{bmatrix}
\begin{gathered}
 1,\tfrac{3}{2}, 2\\
2, 3
\end{gathered}
;z 
\end{bmatrix}
={\,}_2F_1
\begin{bmatrix}
\begin{gathered}
1,\tfrac{3}{2}\\
3
\end{gathered}
;z 
\end{bmatrix},
\end{gather*}
and
\begin{equation*}
\sum_{k=0}^{n}\binom{2k}{k}\frac{1}{1+k}\biggl(\frac{z}{4}\biggr)^k
={\,}_2F_1
\begin{bmatrix}
\begin{gathered}
\tfrac{1}{2}, 1\\
2
\end{gathered}
;z 
\end{bmatrix}
-\frac{(1+2n)!!}{(2+2n)!!}\frac{z^{1+n}}{2+n}
{\,}_3F_2
\begin{bmatrix}
\begin{gathered}
 1, \tfrac{3}{2}+n, 2+n\\
2+n, 3+n
\end{gathered}
;z
\end{bmatrix}.
\end{equation*}
Since the Catalan numbers $C_n$ for $n\in\mathbb{N}_0$ can be analytically computed~\cite{Koshy, Roman} by
\begin{equation*}
C_n=\frac{1}{1+n}\binom{2n}{n},
\end{equation*}
the above three equations can be rearranged as
\begin{gather}\label{2F1-1/2-1-2}
\sum_{k=0}^{\infty}C_k\biggl(\frac{z}{4}\biggr)^k
=
{\,}_2F_1
\begin{bmatrix}
\begin{gathered}
\tfrac{1}{2}, 1\\
2
\end{gathered}
;z 
\end{bmatrix},\\
\sum_{k=0}^{\infty}C_{1+k}\biggl(\frac{z}{4}\biggr)^{k} 
={\,}_3F_2
\begin{bmatrix}
\begin{gathered}
 1,\tfrac{3}{2}, 2\\
2, 3
\end{gathered}
;z 
\end{bmatrix}
={\,}_2F_1
\begin{bmatrix}
\begin{gathered}
1,\tfrac{3}{2}\\
3
\end{gathered}
;z 
\end{bmatrix},\label{notag-eq}
\end{gather}
and
\begin{equation}\label{finite-sum-eq}
\begin{aligned}
\sum_{k=0}^{n}C_k\biggl(\frac{z}{4}\biggr)^k
&={\,}_2F_1
\begin{bmatrix}
\begin{gathered}
\tfrac{1}{2}, 1\\
2
\end{gathered}
;z 
\end{bmatrix}
-\frac{(1+2n)!!}{(2+2n)!!}\frac{z^{1+n}}{2+n}
{\,}_3F_2
\begin{bmatrix}
\begin{gathered}
 1, \tfrac{3}{2}+n, 2+n\\
2+n, 3+n
\end{gathered}
;z
\end{bmatrix}\\
&={\,}_2F_1
\begin{bmatrix}
\begin{gathered}
\tfrac{1}{2}, 1\\
2
\end{gathered}
;z 
\end{bmatrix}
-\frac{(1+2n)!!}{(2+2n)!!}\frac{z^{1+n}}{2+n}
{\,}_2F_1
\begin{bmatrix}
\begin{gathered}
 1, \tfrac{3}{2}+n\\
3+n
\end{gathered}
;z
\end{bmatrix}.
\end{aligned}
\end{equation}
The formula~\eqref{notag-eq} follows from~\eqref{equ:i} with $\alpha=1$ by shifting the
index $k\mapsto k+1$ and dividing by $\frac{z}4$.
Therefore, we can regard the Gauss hypergeometric functions
$
{\,}_2F_1
\begin{bmatrix}
\begin{gathered}
\tfrac{1}{2}, 1\\
2
\end{gathered}
;z 
\end{bmatrix}
$
and
$
{\,}_2F_1
\begin{bmatrix}
\begin{gathered}
1,\tfrac{3}{2}\\
3
\end{gathered}
;z 
\end{bmatrix}
$
as generating functions of the Catalan numbers $C_n$.
\end{rem}

\begin{rem}
In terms of the Catalan numbers $C_n$, we can write
\begin{equation*}
\sum_{k=1}^{\infty}\binom{2k}{k}\frac{1}{\alpha+k}\biggl(\frac{z}{4}\biggr)^k
=\sum_{k=1}^{\infty}C_k\frac{1+k}{\alpha+k}\biggl(\frac{z}{4}\biggr)^k
\end{equation*}
for $\alpha\in\mathbb{C}\setminus\mathbb{Z}^-$ and $|z|\le1$. As a result, the series expansions~\eqref{equ:ii} and~\eqref{equ:i} in Theorem~\ref{Arjun-Gauss-thm} can be regarded as generalizations of the generating function
\begin{equation}\label{Catalan-GenF}
\frac2{1+\sqrt{1-4z}\,}=\frac{1-\sqrt{1-4z}\,}{2z}
=\sum_{n=0}^\infty C_nz^n, \quad |z|<\frac{1}{4}
\end{equation}
in~\cite[p.~621, Entry~26.5.2]{NIST-HB-2010}.
\par
Comparing~\eqref{Catalan-GenF} with~\eqref{2F1-1/2-1-2}, we conclude that
\begin{equation}\label{Entry84p1990}
{\,}_2F_1
\begin{bmatrix}
\begin{gathered}
\tfrac{1}{2}, 1\\
2
\end{gathered}
;z 
\end{bmatrix}
=\frac2{1+\sqrt{1-z}\,}
=\frac{2\bigl(1-\sqrt{1-z}\,\bigr)}{z}, \quad |z|\le1,
\end{equation}
where the extension to $|z|=1$ follows from absolute convergence of the series, since $c-a-b=\frac12>0$; it can be found in~\cite[p.~473, Entry~84]{p1990}.
\par
In the very recent paper~\cite{mahamed-catalan.tex}, Abdel-Latif and Qi derived an explicit sum expressed in terms of the modified Bessel functions of the first kind~\cite[p.~249, Entry~10.25.2]{NIST-HB-2010}
$$
I_\nu(z)= \sum_{k=0}^\infty\frac1{k!\Gamma(\nu+k+1)}\biggl(\frac{z}2\biggr)^{2k+\nu}
$$ 
and the Stirling numbers of the second kind $S(n,k)$, which can be generated~\cite[p.~625, Entry~26.8.12]{NIST-HB-2010} by
\begin{equation*}%\label{2Stirl-funct-rew}
\biggl(\frac{\te^x-1}{x}\biggr)^n=\sum_{k=0}^\infty \frac{S(k+n,n)}{\binom{k+n}{n}} \frac{x^{k}}{k!}, \quad n\ge0,
\end{equation*}
for an infinite series $\sum_{n=0}^{\infty}(n+\alpha)^mC_n\frac{x^n}{n!}$ with $\alpha,x\in\mathbb{R}$, whose coefficient $(n+\alpha)^mC_n$ contains the Catalan numbers $C_n$.
\end{rem}

\begin{rem}
In~\cite[Theorem~22]{Catalan-Int-Surv.tex}, Qi and Guo obtained
\begin{equation}\label{Catalan-Sum}
\sum_{k=0}^{n}\frac{C_k}{4^k}
=\frac{2}{\pi}\biggl[B\biggl(\frac12,\frac12\biggr) -B\biggl(\frac12,\frac32+n\biggr)\biggr], \quad n\in\mathbb{N}_0.
\end{equation}
Comparing~\eqref{Catalan-Sum} with~\eqref{finite-sum-eq} for $z=1$ leads to
\begin{multline}\label{eval-3F2}
{\,}_3F_2
\begin{bmatrix}
\begin{gathered}
 1, \tfrac{3}{2}+n, 2+n\\
2+n, 3+n
\end{gathered}
;1
\end{bmatrix}
={\,}_2F_1
\begin{bmatrix}
\begin{gathered}
 1, \tfrac{3}{2}+n\\
3+n
\end{gathered}
;1
\end{bmatrix}\\
=\frac{2}{\sqrt{\pi}\,}\frac{\Gamma(3+n)}{\Gamma\bigl(\frac{3}{2}+n\bigr)}B\biggl(\frac12,\frac32+n\biggr)
=2(2+n)
\end{multline}
for $n\in\mathbb{N}_0$.
\end{rem}

\begin{rem}
By virtue of~\eqref{2F1-Beta}, when $\Re(\alpha)>0$ and $0\le z\le1$, the equations~\eqref{equ:ii} and~\eqref{equ:B} can be computed by
\begin{equation*}
\sum_{k=1}^{\infty}\binom{2k}{k}\frac{1}{\alpha+k}\biggl(\frac{z}{4}\biggr)^k
=\frac{1}{z^\alpha}B_z\biggl(\alpha, \frac12\biggr)-\frac{1}{\alpha}
\end{equation*}
and
\begin{multline*}
\sum_{k=1}^{n}\binom{2k}{k}\frac{1}{\alpha+k}\biggl(\frac{z}{4}\biggr)^k
=\frac{1}{z^\alpha}B_z\biggl(\alpha,\frac12\biggr)-\frac{1}{\alpha}\\*
-\frac{(1+2n)!!}{(2+2n)!!}\frac{z^{1+n}}{1+\alpha+n}
{\,}_3F_2
\begin{bmatrix}
\begin{gathered}
 1, \tfrac{3}{2}+n, 1+\alpha+n\\
2+n, 2+\alpha+n
\end{gathered}
;z
\end{bmatrix}.
\end{multline*}
\end{rem}

\begin{rem}\label{recovery}
The formula~\eqref{Entry28NIST-HB-2010} can be recovered from the Euler integral representation~\eqref{Euler-Integral-Gauss-HF} of the Gauss hypergeometric function ${\,}_2F_1$ as follows:
\begin{gather*}
{\,}_2F_1
\begin{bmatrix}
\begin{gathered}
a, b\\
1+b
\end{gathered}
;z 
\end{bmatrix}
=\frac{\Gamma(1+b)}{\Gamma(b)\Gamma(1)} \int_{0}^{1}t^{b-1}(1-zt)^{-a}\td t\\
=b\int_{0}^{z}\biggl(\frac{u}{z}\biggr)^{b-1}(1-u)^{-a}\frac{\td u}{z}
=\frac{b}{z^b}\int_{0}^{z}u^{b-1}(1-u)^{-a}\td u
=\frac{b}{z^b} B_z(b,1-a)
\end{gather*}
for $\Re(b)>0$ and $0\le z\le1$.
\end{rem}

\begin{rem}
Letting $z\to1^-$ in~\eqref{3F2-closed1} and~\eqref{3F2-closed2} immediately yields
\begin{align*}
{\,}_3F_2
\begin{bmatrix}
\begin{gathered}
 1,\tfrac{3}{2}, \tfrac{3}{2}+n\\
2, \tfrac{5}{2}+n
\end{gathered}
;1
\end{bmatrix}
&=\frac{2(3+2n)}{1+2n} \biggl[\frac{\pi}{2}\frac{1+2n}{4^{n}} \binom{2n}{n}-1\biggr]
\intertext{and}
{\,}_3F_2
\begin{bmatrix}
\begin{gathered}
 1,\tfrac{3}{2}, 2+n\\
2, 3+n
\end{gathered}
;1
\end{bmatrix}
&=2(2+n)\Biggl[\frac{(-1)^{n}}{4^{n}}\sum_{k=0}^{1+2n} \frac{(-1)^{k}}{2n-2k+1}\binom{1+2n}{k} -\frac{1}{1+n}\Biggr].
\end{align*}
\par
Setting $z=-1$ (the formulas extend to $z=-1$ by continuity, since the series
converge there) and using the branches fixed after Corollary~\ref{2F1-sum-ID},
we have $z^{n+1/2}=(-1)^n\ti$, $\arcsin\ti=\ti\arcsinh1$,
$\sin(\ti x)=\ti\sinh x$ and $\cos(\ti x)=\cosh x$. This gives
\begin{gather*}
\begin{aligned}
{\,}_3F_2
\begin{bmatrix}
\begin{gathered}
 1,\tfrac{3}{2}, \tfrac{3}{2}+n\\
2, \tfrac{5}{2}+n
\end{gathered}
;-1
\end{bmatrix}
&=-\frac{2(3+2n)}{1+2n}\Biggl(\frac{1+2n}{4^{n}} \frac{(-1)^n}{\ti}\Biggl[\binom{2n}{n} \arcsin\ti\\
+(-1)^n\Biggl(\sum_{j=0}^{n-1}&+\sum_{j=1+n}^{2n}\Biggr)(-1)^{j}\binom{2n}{j} \frac{\sin[2(n-j)\arcsin\ti]}{2(n-j)}\Biggr]-1\Biggr)
\end{aligned}\\
\begin{aligned}
&=-\frac{2(3+2n)}{1+2n}\Biggl(\frac{(-1)^n(1+2n)}{4^{n}} \Biggl[\binom{2n}{n} \arcsinh1\\
&\quad+(-1)^n\Biggl(\sum_{j=0}^{n-1}+\sum_{j=1+n}^{2n}\Biggr)(-1)^{j}\binom{2n}{j} \frac{\sinh[2(n-j)\arcsinh1]}{2(n-j)}\Biggr]-1\Biggr)
\end{aligned}\\
\begin{aligned}
&=\frac{2(3+2n)}{1+2n}\Biggl((-1)^{1+n}\frac{1+2n}{4^{n}} \Biggl[\binom{2n}{n} \arcsinh1\\
&\quad+(-1)^n\Biggl(\sum_{j=0}^{n-1}+\sum_{j=1+n}^{2n}\Biggr)(-1)^{j}\binom{2n}{j} \frac{\sinh[2(n-j)\arcsinh1]}{2(n-j)}\Biggr]+1\Biggr)
\end{aligned}
\end{gather*}
and
\begin{multline*}
{\,}_3F_2
\begin{bmatrix}
\begin{gathered}
 1,\tfrac{3}{2}, 2+n\\
2, 3+n
\end{gathered}
;-1
\end{bmatrix}\\
=
\frac{2(2+n)}{1+n}\Biggl[1-\frac{1+n}{4^{n}} \sum_{j=0}^{1+2n}(-1)^{j} \binom{1+2n}{j} \frac{\cos[(2n-2j+1)\arcsin\ti]-1}{2n-2j+1}\Biggr]\\
=
\frac{2(2+n)}{1+n}\Biggl[1-\frac{1+n}{4^{n}} \sum_{j=0}^{1+2n}(-1)^{j} \binom{1+2n}{j} \frac{\cosh[(2n-2j+1)\arcsinh1]-1}{2n-2j+1}\Biggr]
\end{multline*}
for $n\in\mathbb{N}_0$.
\end{rem}

\begin{rem}
Taking $z\to1^-$ in~\eqref{closed=integer} leads to
\begin{multline*}
\sum_{k=1}^{\infty}\binom{2k}{k}\frac{1}{1+k+n}\frac{1}{4^k}\\
\begin{aligned}
&=\frac{(-1)^{1+n}}{4^{n}} \sum_{j=0}^{1+2n}(-1)^{j} \binom{1+2n}{j} \frac{\cos[(2n-2j+1)\arcsin1]-1}{2n-2j+1}-\frac{1}{1+n}\\
&=\frac{(-1)^{n}}{4^{n}} \sum_{j=0}^{1+2n} \binom{1+2n}{j} \frac{(-1)^{j}}{2n-2j+1}-\frac{1}{1+n},
\end{aligned}
\end{multline*}
that is,
\begin{equation}\label{closed=i}
\sum_{k=1}^{\infty}\binom{2k}{k}\frac{1}{1+k+n}\frac{1}{4^k}
=\frac{(-1)^{n}}{4^{n}} \sum_{k=0}^{1+2n} \frac{(-1)^{k}}{2n-2k+1}\binom{1+2n}{k} -\frac{1}{1+n}
\end{equation}
for $n\in\mathbb{N}_0$.
\par
Further setting $z=-1$ in~\eqref{closed=integer} (allowed by continuity, since the
series converges for $|z|\le1$) and using the branches fixed after
Corollary~\ref{2F1-sum-ID}, so that $\arcsin\sqrt{-1}=\ti\arcsinh1$ and
$\cos(\ti x)=\cosh x$, results in
\begin{multline*}
\sum_{k=1}^{\infty}\binom{2k}{k}\frac{1}{1+k+n}\frac{(-1)^k}{4^k}\\*
\begin{aligned}
&=\frac{1}{4^{n}}\sum_{j=0}^{1+2n}(-1)^{j} \binom{1+2n}{j} \frac{\cos[(2n-2j+1)\arcsin\sqrt{-1}\,]-1}{2n-2j+1}-\frac{1}{1+n}\\
&=\frac{1}{4^{n}}\sum_{j=0}^{1+2n}(-1)^{j} \binom{1+2n}{j} \frac{\cos[\ti(2n-2j+1)\arcsinh1]-1}{2n-2j+1}-\frac{1}{1+n}\\
&=\frac{1}{4^{n}}\sum_{j=0}^{1+2n}(-1)^{j} \binom{1+2n}{j} \frac{\cosh[(2n-2j+1)\arcsinh1]-1}{2n-2j+1}-\frac{1}{1+n}
\end{aligned}
\end{multline*}
for $n\in\mathbb{N}_0$. In summary, we have
\begin{multline*}
\sum_{k=1}^{\infty}\binom{2k}{k}\frac{1}{1+k+n}\frac{(-1)^k}{4^k}\\
=\frac{1}{4^{n}}\sum_{k=0}^{1+2n}(-1)^{k} \binom{1+2n}{k} \frac{\cosh[(2n-2k+1)\arcsinh1]-1}{2n-2k+1}-\frac{1}{1+n}.
\end{multline*}
\par
Similarly, setting $z=1$ and $z=-1$ in~\eqref{closed=int} respectively
(with $z^{n+1/2}=(-1)^n\ti$ at $z=-1$, by the branches fixed after
Corollary~\ref{2F1-sum-ID}) leads to
\begin{equation*}
\sum_{k=1}^{\infty}\binom{2k}{k}\frac{1}{2k+2n+1}\frac{1}{4^k}
=\frac{\pi}{2^{1+2n}}\binom{2n}{n}-\frac{1}{1+2n}
\end{equation*}
and
\begin{multline*}
\sum_{k=1}^{\infty}\binom{2k}{k}\frac{1}{2k+2n+1}\frac{(-1)^k}{4^k}
=\frac{(-1)^n}{4^{n}}\Biggl[\binom{2n}{n} \arcsinh1\\
+(-1)^n\Biggl(\sum_{k=0}^{n-1}
+\sum_{k=1+n}^{2n}\Biggr) (-1)^{k}\binom{2n}{k} \frac{\sinh[2(n-k)\arcsinh1]}{2(n-k)}\Biggr]
-\frac{1}{1+2n}.
\end{multline*}
\end{rem}

\begin{rem}
Comparing~\eqref{equ:vVV} with~\eqref{closed=i} and simplifying yield a simple combinatorial identity
\begin{equation}\label{Comb-ID}
\sum_{k=0}^{1+2n} \frac{(-1)^{k}}{2n-2k+1}\binom{1+2n}{k}
=(-1)^n2^{1+2n}\frac{(2n)!!}{(1+2n)!!}, \quad n\in\mathbb{N}_0.
\end{equation}
\par
It is interesting to mention that the combinatorial identity~\eqref{Comb-ID} can also be established directly by employing the Gauss summation theorem~\eqref{hyperg-summ-thm}.
\end{rem}

\begin{rem}
On 10 April 2026, at the website \url{https://math.stackexchange.com/q/5132286}, the last author asked the following question:
What is the closed form of the Gauss hypergeometric function ${\,}_2F_1
\begin{bmatrix}
\begin{gathered}
\tfrac{1}{2}, a\\
 1+a
\end{gathered}
;z 
\end{bmatrix}$
for $\Re(a)>-1$ and $|z|\le1$? By 12 April 2026, there were several answers to this question, all of which are useful for this paper.
\end{rem}

\section{Conclusions}
In this paper, we mainly obtained the following results:
\begin{enumerate}
\item
The claims~\eqref{anonymous1} and~\eqref{anonymous2} were confirmed by Theorems~\ref{Arjun-Gauss-thm} and~\ref{Finite-sum-thm}. See Remark~\ref{rem1}.
\item
The evaluations of the Gauss hypergeometric function ${\,}_2F_1
\begin{bmatrix}
\begin{gathered}
\tfrac{1}{2}, \alpha\\
 1+\alpha
\end{gathered}
;z 
\end{bmatrix}$
for $\alpha\in\mathbb{C}\setminus\mathbb{Z}^-$ and $|z|\le1$ were discussed. These evaluations include the formulas~\eqref{2F1n+half-2} and~\eqref{2F1=1+n}, the recovery of the formula~\eqref{Entry28NIST-HB-2010} in Remark~\ref{recovery}, the formulas~\eqref{1+n2+n} and~\eqref{Entry84p1990}, the equalities in~\eqref{eval-3F2}.
\item
The incomplete beta functions $B_z\bigl(\frac12+n,\frac{1}{2}\bigr)$ and $B_z\bigl(1+n,\frac12\bigr)$ are elementary integrals of $\sin^{2n}\theta$ and $\sin^{2n+1}\theta$, whose antiderivatives are classical (see, for example,~\cite[Entry~2.513]{Gradshteyn-Ryzhik-Table-8th}). The formulas~\eqref{beta-closed1} and~\eqref{beta-closed2} give these antiderivatives in a uniform form, as finite sums of $\sin[2(n-k)\theta]$ and $\cos[(2n-2k+1)\theta]$ with binomial coefficients, which we did not find stated in this form in the literature.
\item
The derivative formula~\eqref{deriv-form} appears to be new to the literature.
\item
The closed forms~\eqref{final-2F1}, \eqref{2F1arccos}, \eqref{2F1arcsin}, \eqref{2F1izan-peraz}, and~\eqref{2F1(izan-peraz)-m} are obtained from Corollary~\ref{2F1-sum-ID} and from closed forms in the literature by the Pfaff transformation.
\item
The combinatorial identity~\eqref{Comb-ID} is interesting too.
\end{enumerate}
\par
This paper is a revised version of the arXiv preprint at~\url{https://doi.org/10.48550/arXiv.2608.05193}.

\section{Declarations}

\paragraph{\bf Authors' Contributions}
All authors contributed equally to the manuscript and read and approved the final manuscript.

\paragraph{\bf Funding}
The author F. Qi was partially supported by the Natural Science Foundation of Inner Mongolia Autonomous Region (Grant No.~2025QN01041) and by the Youth Project of Hulunbuir City for Basic Research and Applied Basic Research (Grant No.~GH2024020).

\paragraph{\bf Institutional Review Board Statement}
Not applicable.

\paragraph{\bf Informed Consent Statement}
Not applicable.

\paragraph{\bf Ethical Approval}
The research does not involve human or animal subjects.

\paragraph{\bf Availability of Data and Material}
Data sharing is not applicable to this article as no new data were created or analyzed in this study.

\paragraph{\bf Competing Interests}
The authors declare that they have no competing interests.

\paragraph{\bf Use of AI Tools Declaration}
The authors declare they have not used Artificial Intelligence (AI) tools in the creation of this article.

\paragraph{\bf Acknowledgements}
Not applicable.


\begin{thebibliography}{99}

\bibitem{mahamed-catalan.tex}
M. S. Abdel-Latif and F. Qi, \textit{Summing weighted Catalan series using modified Bessel functions and Stirling numbers}, Ukrainian Math. J. (2027), to appear.

\bibitem{abram}
M. Abramowitz and I. A. Stegun (Eds.), \textit{Handbook of Mathematical Functions with Formulas, Graphs, and Mathematical Tables}, National Bureau of Standards, Applied Mathematics Series \textbf{55}, 10th printing, Washington, 1972.

\bibitem{a2000}
G. E. Andrews, R. Askey, and R. Roy, \textit{Special Functions}, Encyclopedia of Mathematics and its Applications \textbf{71}, Cambridge University Press, Cambridge, 1999. DOI: \url{https://doi.org/10.1017/CBO9781107325937}.

\bibitem{b1935}
W. N. Bailey, \emph{Generalized Hypergeometric Series}, Cambridge Tracts in Mathematics and Mathematical Physics, No.~32. Stechert-Hafner, Inc., New York, 1964.

\bibitem{closed-form-what-why-care}
J. M. Borwein and R. E. Crandall, \emph{Closed forms: what they are and why we care}, Notices Amer. Math. Soc. \textbf{60} (2013), no.~1, 50\nobreakdash--65. DOI: \url{https://doi.org/10.1090/noti936}.

\bibitem{Driver-Johnston-2006}
K. A. Driver and S. J. Johnston, \textit{An integral representation of some hypergeometric functions}, Electron. Trans. Numer. Anal. \textbf{25} (2006), 115\nobreakdash--120.

\bibitem{AADM-3316.tex}
C. Fr\'etigny and F. Qi, \emph{Higher-order derivatives of two arctangent-based functions and their relation to Gauss hypergeometric functions}, Appl. Anal. Discrete Math. \textbf{20} (2026), no.~2, 298\nobreakdash--317. DOI: \url{https://doi.org/10.2298/AADM260604032F}.

\bibitem{Gradshteyn-Ryzhik-Table-8th}
I. S. Gradshteyn and I. M. Ryzhik, \emph{Table of Integrals, Series, and Products}, Translated from the Russian, Translation edited and with a preface by Daniel Zwillinger and Victor Moll, Eighth edition, Revised from the seventh edition, Elsevier/Academic Press, Amsterdam, 2015. DOI: \url{https://doi.org/10.1016/B978-0-12-384933-5.00013-8}.

\bibitem{Koshy}
T. Koshy, \emph{Catalan Numbers with Applications}, Oxford University Press, Oxford, 2009.

\bibitem{axioms-2962911.tex}
Y.-W. Li and F. Qi, \emph{A new closed-form formula of the Gauss hypergeometric function at specific arguments}, Axioms \textbf{13} (2024), no.~5, Art.~317, 24~pp. DOI: \url{https://doi.org/10.3390/axioms13050317}.

\bibitem{Gauss-Milovanovic-Qi.tex}
G. V. Milovanovi\'c and F. Qi, \textit{Closed-form formulas of two Gauss hypergeometric functions of specific parameters}, J. Math. Anal. Appl. \textbf{543} (2025), no.~2, Part~3, Paper No.~129024, 35~pp. DOI: \url{https://doi.org/10.1016/j.jmaa.2024.129024}.

\bibitem{NIST-HB-2010}
F. W. J. Olver, D. W. Lozier, R. F. Boisvert, and C. W. Clark (eds.), \emph{NIST Handbook of Mathematical Functions}, Cambridge University Press, New York, 2010. URL: \url{http://dlmf.nist.gov/}.

\bibitem{p1990}
A. P. Prudnikov, Yu. A. Brychkov, and O. I. Marichev, \emph{Integrals and Series}, Vol.~3. \emph{More Special Functions}. Translated from the Russian by G. G. Gould. Gordon and Breach Science Publishers, New York, 1990.

\bibitem{Gauss-hyperg-Int.tex}
F. Qi, \textit{An integral representation of the Gauss hypergeometric functions and its applications}, Analysis \textbf{45} (2025), no.~4, 279\nobreakdash--290. DOI: \url{https://doi.org/10.1515/anly-2025-0066}.

\bibitem{Qi-Wilf.tex}
F. Qi, \textit{A power series expansion of the Wilf function},  arXiv:2110.08576v4. DOI: \url{https://doi.org/10.48550/arXiv.2110.08576}.

\bibitem{Gauss-2Formula.tex}
F. Qi, \textit{Combinatorial identities derived from explicit formulas of Gauss hypergeometric functions}, arXiv:2607.10643. DOI: \url{https://doi.org/10.48550/arXiv.2607.10643}.

\bibitem{Catalan-Int-Surv.tex}
F. Qi and B.-N. Guo, \textit{Integral representations of the Catalan numbers and their applications}, Mathematics \textbf{5} (2017), no.~3, Art.~40, 31~pp. DOI: \url{https://doi.org/10.3390/math5030040}.

\bibitem{Qi-Izan-Peraz-Gauss.tex}
F. Qi, C.-Y. He, and D. Lim, \textit{Explicit formulas of two Gauss hypergeometric functions and several combinatorial identities}, Discrete Appl. Math. \textbf{393} (2026), 215\nobreakdash--229. DOI: \url{https://doi.org/10.1016/j.dam.2026.06.023}.

\bibitem{q2022}
F. Qi and D. Lim, \textit{Integral representations and properties of several finite sums containing central binomial coefficients}, ScienceAsia \textbf{49} (2023), no.~2, 205\nobreakdash--211. DOI: \url{http://dx.doi.org/10.2306/scienceasia1513-1874.2022.137}.

\bibitem{r1960}
E. D. Rainville, \emph{Special Functions}, Macmillan, New York, 1960.

\bibitem{Roman}
S. Roman, \emph{An Introduction to Catalan Numbers}, with a foreword by Richard Stanley. Compact Textbook in Mathematics. Birkh\"auser/Springer, Cham, 2015. DOI: \url{https://doi.org/10.1007/978-3-319-22144-1}.

\bibitem{s1966}
L. J. Slater, \emph{Generalized Hypergeometric Functions}, Cambridge University Press, Cambridge, 1966.

\bibitem{Temme-96-book}
N. M. Temme, \emph{Special Functions: An Introduction to Classical Functions of Mathematical Physics}, A Wiley-Interscience Publication, John Wiley \& Sons, Inc., New York, 1996. DOI: \url{https://doi.org/10.1002/9781118032572}.

\end{thebibliography}
\end{document}